\documentclass[12pt]{smfart}
\usepackage{color}
\usepackage{amssymb,verbatim}
\usepackage{amsthm,array,amssymb,amscd,amsfonts,amssymb,latexsym, url}
\usepackage{amsmath}
\usepackage{amscd}
\usepackage{pb-diagram}
\usepackage[all]{xy}
\usepackage[french]{babel}
\usepackage{times}

 \usepackage{mathrsfs}  
\usepackage{bm}        
\usepackage{mathtools} 

 \newcounter{spec}
{\end{list}}

\renewcommand{\P}{{\mathbf P}}

\newcommand{\sD}{{\mathcal D}}

\newcommand{\N}{{\mathbb N}}

\newcommand{\Z}{{\mathbb Z}}
\newcommand{\Q}{{\mathbb Q}}
\newcommand{\C}{{\mathbb C}}

\renewcommand{\L}{{\mathbb L}}

\newcommand{\oi}{\hskip1mm {\buildrel \simeq \over \rightarrow} \hskip1mm}

\newcommand{\Br}{{\operatorname{Br}}}
\newcommand{\Ker}{{\operatorname{Ker}}}
\newcommand{\Coker}{{\operatorname{Coker}}}
\newcommand{\Hom}{{\operatorname{Hom}}}
\newcommand{\Spec}{{\operatorname{Spec \ }}}

\renewcommand{\div}{{\operatorname{Div}}}
\newcommand{\tors}{{\operatorname{tors}}}

\renewcommand{\lim}{\varprojlim}

\renewcommand{\phi}{\varphi}

\numberwithin{equation}{section}

\newfont{\gothic}{eufb10}

\renewcommand{\qed}{{\hfill$\square$}}

\newtheorem{theo}{Th\'{e}or\`{e}me}[section]
\newtheorem{prop}[theo]{Proposition}

\newtheorem{lem}[theo]{Lemme}
\newtheorem{cor}[theo]{Corollaire}
\theoremstyle{definition}
\newtheorem{defi}[theo]{D\'efinition}
\theoremstyle{remark}
\newtheorem{rema}[theo]{Remarque}
\newtheorem{ex}[theo]{Exemple}

\newtheorem{remas}[theo]{Remarques}

\newcommand{\bthe}{\begin{theo}}
\newcommand{\ble}{\begin{lem}}
\newcommand{\bpr}{\begin{prop}}
\newcommand{\bco}{\begin{cor}}
\newcommand{\bde}{\begin{defi}}
\newcommand{\ethe}{\end{theo}}
\newcommand{\ele}{\end{lem}}
\newcommand{\epr}{\end{prop}}
\newcommand{\eco}{\end{cor}}
\newcommand{\ede}{\end{defi}}

\newcommand{\Gal}{{\rm{Gal}}}

\newcommand{\Pic}{\operatorname{Pic}}

\DeclareFontFamily{U}{wncy}{}
\DeclareFontShape{U}{wncy}{m}{n}{%
<5>wncyr5%
<6>wncyr6%
<7>wncyr7%
<8>wncyr8%
<9>wncyr9%
<10>wncyr10%
<11>wncyr10%
<12>wncyr6%
<14>wncyr7%
<17>wncyr8%
<20>wncyr10%
<25>wncyr10}{}
\DeclareMathAlphabet{\cyr}{U}{wncy}{m}{n}
\def\A{{\mathbb A}}

\def\G{{\mathbb G}}

\def \M{{\mathcal M}}
\def \spec {{\rm Spec}\, }
\def \ov{\overline}

\def \fraco{{\rm Frac \,}}
\def \Ga{{\Gamma}}

\def \ext{{\rm Ext}}
\def \calo{{\mathcal O}}

\def\calf{{\mathcal F}}
\def \calt{{\mathcal T}}

\def\T{{\mathcal T}}
\def \Sha{{\cyr{X}}}

\def \cyc{{\rm cyc}}

\renewcommand\theenumi{\it\alph{enumi}}
\renewcommand\labelenumi{\theenumi)}

\begin{document}

\title[Dualit\'e pour les courbes sur $\C((t))$]
 {Dualit\'e et principe local-global pour les tores sur une courbe au-dessus de  $\C((t))$ }
\author{Jean-Louis Colliot-Th\'el\`ene}
\address{C.N.R.S., Universit\'e Paris Sud\\Math\'ematiques, B\^atiment 425\\91405 Orsay Cedex\\France}
\email{jlct@math.u-psud.fr}
\author{David Harari}
\address{Universit\'e Paris Sud\\Math\'ematiques, B\^atiment 425\\91405 Orsay Cedex\\France}
\email{David.Harari@math.u-psud.fr}
\date{16 avril 2014}
\maketitle

 \begin{abstract}
 Pour  $K$ un corps global (corps de nombres ou corps de fonctions d'une variable
sur un corps fini $F$),  on dispose de th\'eor\`emes de dualit\'e
classiques  (Tate, Poitou, Nakayama) pour la cohomologie galoisienne \`a valeurs dans des
tores et des modules ab\'eliens finis. 

Nous \'etablissons de tels th\'eor\`emes pour  $K$  le   corps des fonctions
d'une courbe projective et lisse sur le  corps $F=\C((t))$ des s\'eries formelles en une variable sur 
le corps des complexes. Cela permet de contr\^oler le d\'efaut du principe
de Hasse et l'approximation faible  pour  les espaces homog\`enes sous un tore.

 Il y a ici des diff\'erences avec le cas classique ($K$ corps global), et aussi avec
 le cas r\'ecemment \'etudi\'e o\`u $K$ est un corps de fonctions d'une variable sur un
 corps $p$-adique. Par exemple, la $K$-rationalit\'e d'un  tore n'implique pas ici
 la validit\'e du principe de Hasse pour  ses espaces principaux homog\`enes.
  \end{abstract}

\begin{altabstract}
Over  a global field $K$ (number field, or function field of a curve
over a finite field $F$), arithmetic duality theorems for the
Galois cohomology of tori and finite Galois modules have long been known.
More recent work investigates the case where $K$ is the   function fields of of a curve  over
a $p$-adic field.

For $K$ the function field of a curve over the formal series  field $F=\C((t))$, we
establish analogous duality theorems. We thus control the obstruction
to the local-global principle and to weak approximation for homogeneous
spaces of tori.

There are differences with the
afore described cases. For example the Hasse principle need not hold
for principal homogeneous spaces of a $K$-rational torus.
\end{altabstract}

\section*{Introduction}

Soient $k$ le corps des fractions d'un anneau de valuation discr\`ete, de corps r\'esiduel $\kappa$,  et
 $K$ le corps des fonctions d'une courbe $C$ projective, lisse, g\'eom\'etriquement int\`egre  sur $k$. Dans plusieurs travaux r\'ecents,
 on s'est int\'eress\'e au principle local-global pour les points rationnels des espaces
 principaux homog\`enes (torseurs) sous un $k$-groupe lin\'eaire $G$.
 Citons ici les travaux de Harbater, Hartmann et Krashen \cite{HHK1, HHK2, HHK3}, 
 ceux de Parimala, Suresh et l'un des auteurs du pr\'esent article \cite{CTPaSu}, et ceux de 
 l'autre auteur avec Szamuely \cite{HaSz} et avec Scheiderer et Szamuely \cite{HaSchSz}.
 
 On peut s'int\'eresser au principe local-global pour au moins deux  familles de compl\'etions $K_{v}$
 de $K$ par rapport \`a des valuations discr\`etes de rang 1.
  
 (i) Les compl\'etions par rapport aux valuations triviales sur $k$, c'est-\`a-dire les valuations correspondant
 aux  points ferm\'es de la $k$-courbe $C$.
 
 (ii) Les compl\'etions  par rapport \`a toutes les valuations sur $K$, y compris celles qui induisent la valuation
 du corps complet $k$.
 
 Harbater, Hartmann et Krashen ont montr\'e que si le $K$-groupe $G$ est connexe et $K$-rationnel,
 alors  dans le cas (ii) on a un principe local-global pour les torseurs sous $G$.

Supposons d\'esormais $\kappa=\C$ et $k=\C((t))$ (plus g\'en\'eralement 
tous les r\'esultats 
de cet article sont valables en rempla\c cant $\C$ par un corps 
alg\'ebriquement clos de caract\'eristique z\'ero).

 Dans ce cas $K$ est un corps de dimension cohomologique 2
 et pr\'esente des analogies
avec les corps  globaux de caract\'eristique positive, i.e. avec les corps de fonctions d'une variable sur un corps fini. 
Une diff\'erence importante est l'absence d'un th\'eor\`eme de Tchebotarev sur un tel corps.
Ainsi  un \'el\'ement de $K$ peut par exemple  \^etre un carr\'e dans tous les compl\'et\'es $K_{v}$ pour $v\in \Omega$
sans \^etre un carr\'e dans $K$.

Dans le cas (i), {\it le principe local-global ne vaut pas en g\'en\'eral
 pour les points rationnels des torseurs sous un $K$-tore $K$-rationnel} 
(voir l'exemple~\ref{contreexrationnel})  
 \`a la diff\'erence de ce qui se passe sur un corps de nombres 
(Voskresenski\v {\i}, voir \cite{sansuc}, Cor. 9.7)
ou sur 
le corps des fonctions d'une courbe $p$-adique (\cite{HaSz}, Cor. 5.7).
Autrement dit, avec les notations ci-dessous, on peut avoir 
$\Sha^1(K,T) \neq 0$ m\^eme pour un tore $T$ qui est $K$-rationnel.
 Dans \cite{CTPaSu}, on donne aussi des exemples 
 de $K$-tore alg\'ebrique  -- non $K$-rationnel  -- pour lequel le principe local-global  pour les points rationnels des torseurs ne vaut  pas dans le cas 
(ii).
 
 Au vu des r\'esultats classiques sur un corps  global comme ceux de 
Voskresenski\v {\i} (\cite{vosk}) et Sansuc (\cite{sansuc}), 
 on se demande alors :
 peut-on contr\^oler le d\'efaut du principe de Hasse
 (et aussi de l'approximation faible)  au moyen d'un 
accouplement cohomologique ?

Soit $\Omega$ l'ensemble des points ferm\'es de la $k$-courbe $C$.
Pour tout $K$-sch\'ema en groupes commutatif $G$, et tout entier $i \geq 0,$ on d\'efinit le groupe  
ab\'elien
$$\cyr{X}^{i}(K,G)= \Ker [H^{i}(K,G) \to \prod_{v \in \Omega} H^{i}(K_{v},G)].$$Si $K$ est sous-entendu, on abr\`egera parfois $\Sha^i(K,G)$ en 
$\Sha^i(G)$. 

Soit $\mu_n$ le groupe des racines $n$-i\`eme 
de l'unit\'e dans $k$, on note $\Q/\Z(-1)$ la limite inductive
des $\Z/n(-1):=\Hom(\mu_n,\Z/n)$. C'est un groupe ab\'elien 
isomorphe (non canoniquement) \`a $\Q/\Z$.

Les principaux r\'esultats de dualit\'e 
obtenus dans cet article sont rassembl\'es dans le th\'eor\`eme suivant,
\'etabli au \S 7.

\medskip

{\bf Th\'eor\`eme.} {\it Soient $T$ un $K$-tore 
et $\widehat{T}$ son groupe des caract\`eres.
Il existe des accouplements  de groupes de torsion
$$\cyr{X}^1(K,T) \times \cyr{X}^2(K,\widehat{T}) \to \Q/\Z(-1)$$
et
$$\cyr{X}^1(K,\widehat{T}) \times \cyr{X}^2(K,T) \to \Q/\Z(-1)$$
qui par passage au quotient par les sous-groupes divisibles maximaux
induisent des accouplements parfaits de groupes finis
$$\cyr{X}^1(K,T) \times \cyr{X}^2(K,\widehat{T})/{\rm Div} \to \Q/\Z(-1)$$
et 
$$\cyr{X}^1(K,\widehat{T}) \times \cyr{X}^2(K,T)/{\rm Div} \to \Q/\Z(-1).$$}

Il s'agit en principe d'adapter les arguments donn\'es dans \cite{HaSz}.
Dans \cite{HaSz}, on s'int\'eresse au cas o\`u $K$ est le corps des fonctions d'une courbe
sur un corps $p$-adique.
Dans ce cas-l\`a, il y  a un seul type de dualit\'e \`a \'etudier, \`a savoir
$$\cyr{X}^1(K,T) \times \cyr{X}^2(K,T') \to \Q/\Z,$$
o\`u $T'$ est le $K$-tore dual de $T$. En effet  les r\^oles de $T$ et
$T'$ sont interchangeables.
Ici il nous faut \'etudier deux accouplements distincts (comme dans le
cas des corps globaux).

\medskip

Une cons\'equence (Cor.~8.3) du th\'eor\`eme pr\'ec\'edent est que comme 
dans le cas des corps de nombres, {\it l'obstruction au principe de Hasse 
pour un espace principal homog\`ene $Y$ d'un $K$-tore est contr\^ol\'ee 
par un certain sous-groupe du groupe de Brauer alg\'ebrique $\Br_1 Y$}, 
constitu\'e des \'el\'ements dont la restriction \`a $\Br (Y \times_K K_v)$ 
est constante en toute place $v \in C^{(1)}$. 

On verra aussi (Proposition~\ref{contrexwa}) 
qu'\`a la diff\'erence du cas des corps globaux, {\it il existe  
des contre-exemples \`a l'approximation faible pour un tore
qui ne sont pas contr\^ol\'es par le groupe de Brauer non ramifi\'e du tore}.
Comme la dimension cohomologique de $K$ est 2,
  on ne peut ici esp\'erer utiliser comme substitut du groupe
de Brauer  le groupe $H^3$ non ramifi\'e,
comme le font Scheiderer,  Szamuely \ et le second auteur \cite{HaSchSz}
sur les  corps de fonctions d'une variable sur un corps $p$-adique.

N\'eanmoins on verra qu'{\it on a encore une suite exacte 
\`a la Voskresenski\v {\i}-Sansuc 
d\'ecrivant le d\'efaut d'approximation faible} (Th\'eor\`eme~\ref{vosksansuc}).

\medskip

Il resterait \`a comprendre s'il y a une relation entre les obstructions au principe
local-global par rapport aux places de $K$ triviales sur $k$, comme \'etudi\'ees
dans le pr\'esent article, et les obstructions sup\'erieures, par rapport \`a toutes
les valuations  de $K$, et faisant intervenir des lois de r\'eciprocit\'e aux points de codimension 2
sur les surfaces
arithm\'etiques, utilis\'ees dans \cite{CTPaSu}.

\medskip

{\bf Notations}

Soit $A$ un groupe ab\'elien (qu'on supposera toujours \'equip\'e 
de la topologie discr\`ete si on ne d\'efinit pas une autre topologie dessus). 
Pour $n>0$ entier on note ${}_{n}A \subset A$
le sous-groupe des \'el\'ements annul\'es par $n$. Pour $\ell$ premier, on note
$A\{l\} \subset A$ la torsion $\ell$-primaire de $A$. On note $A^{(\ell)}$  la limite 
projective des $A/{\ell}^nA$. On note $A^D=\Hom_c(A,\Q/Z(-1))$ le groupe des
homomorphismes continus de $A$ dans $\Q/\Z(-1)$.
Le foncteur $(.)^D$ est exact sur la cat\'egorie
des groupes discrets, et il r\'ealise de plus une anti-\'equivalence 
de cat\'egories entre les groupes discrets de torsion et les groupes 
profinis.

Dans tout ce texte, la cohomologie utilis\'ee est la cohomologie \'etale
(ou galoisienne si on travaille sur un corps). 
Soit $C$ une courbe projective et lisse sur un corps $k$ (dont on notera
$\bar k$ une cl\^oture alg\'ebrique); soient $U$ 
un ouvert de Zariski non vide de $C$ et $j_U : U \to C$ l'inclusion.
Pour tout faisceau \'etale  $\calf$ sur $U$,
on note $(j_U)_{!}F$ le prolongement de $\calf$ par z\'ero 
et pour tout $i\geq 0$ entier on note
$$H^{i}_{c}(U,\calf):= H^{i}(C,(j_U)_{!}\calf)$$
les groupes de cohomologie \'etale \`a support compact (cf. 
\cite{milneEC}, Prop. III.1.29).

\smallskip

{\bf Remerciements.} Les auteurs remercient le Centre Interfacultaire Bernoulli
de l'\'Ecole Polytechnique F\'ed\'erale de Lausanne pour son hospitalit\'e
\`a l'automne 2012, ainsi que R. Parimala et T. Szamuely pour 
d'int\'eressantes discussions. Pour ce travail, le premier auteur
a b\'en\'efici\'e d'une aide de l'Agence 
Nationale de la Recherche portant la r\'ef\'erence ANR-12-BL01-0005.

\section{Les corps}\label{lescorps}

Soient  $k=\C((t))$ et   $K=\C((t))(C)$  le corps des fonctions d'une courbe
$C$ projective, lisse, g\'eom\'etriquement connexe sur $k$.  
Dans la terminologie de \cite{Parimala}, ce sont
des corps (lg).

Pour tout point ferm\'e $v$
de $C$, le corps r\'esiduel est une extension finie de $\C((t))$, donc de la forme $\C((s))$
avec $s=t^{1/n}$ pour $n$ convenable, et le
 corps des fractions du compl\'et\'e de l'anneau local de $C$ en $v$
est isomorphe \`a $\C((s))((u))$. 

Une extension finie d'un corps $K=\C((t))(C)$ est de la forme $\C((t'))(C')$  pour 
$C'$ une courbe g\'eom\'etriquement connexe convenable sur un corps $\C((t'))$.

Une extension finie d'un corps $\C((s))((u))$ est de la forme $\C((s'))((u'))$.

Un corps $K$ de la forme $\C((t))(C)$ ou $\C((s))((u))$ a les propri\'et\'es suivantes.
Toute extension finie de 
$K$  est encore du m\^eme  type.
Le corps $K$ est  un corps $C_{2}$ -- puisque $\C((t))$ est un corps $C_{1}$.
  Le corps $K$  est de dimension cohomologique 2.
L'extension ab\'elienne maximale de $K$ est un corps $C_{1}$, donc de dimension cohomologique 1.
Ceci r\'esulte simplement du fait que la cl\^oture alg\'ebrique de $\C((t))$ est, d'apr\`es le th\'eor\`eme de
Puiseux, une extension ab\'elienne de $\C((t))$.

Pour toute alg\`ebre simple centrale sur $K$ , la norme r\'eduite est surjective : ceci r\'esulte du fait
que $K$ est un corps $C_{2}$.

{\it Pour toute alg\`ebre simple centrale $D$ sur $K$, indice et exposant de l'alg\`ebre $D$  co\"{\i}ncident.}
 Pour la partie 2-primaire et la partie 3-primaire, c'est une cons\'equence connue (M. Artin) de la propri\'et\'e $C_{2}$.    
 Pour le cas local, il suffit de renvoyer \`a \cite[Thm. 1.5]{CTGiPa}.
Pour le cas $\C((t))(C)$, comme le mentionne Parimala \cite{Parimala},  cela peut se voir soit  en utilisant la m\'ethode de scindage de
la ramification de Saltman \cite{Saltman}, soit 
 par la m\'ethode de recollement de corps de Harbater, Hartmann et Krashen
\cite[Thm. 5.5]{HHK1}.

\medskip

Soit $L/K$ une extension finie galoisienne de  $K=\C((t))(C)$. 
Ceci correspond \`a un rev\^etement fini \'eventuellement
ramifi\'e de courbes int\`egres r\'eguli\`eres $D \to C$ propres au-dessus du corps $k=\C((t))$. Pour $v \in C^{(1)}$
un point ferm\'e o\`u $D \to C$ est \'etale, et $w\in D^{(1)}$ point ferm\'e au-dessus de $v$, le groupe
de d\'ecomposition s'identifie au groupe de Galois des extensions de corps r\'esiduelles. 
Comme toute extension  finie de $\C((t))$, cette extension est cyclique. Ainsi presque 
tous les groupes de d\'ecomposition sont cycliques, comme dans le cas des extensions galoisiennes de corps globaux.

Par contre {\it  on n'a pas l'analogue du th\'eor\`eme de Tchebotarev} :
Pour une extension $L/K$ cyclique, il se peut qu'aucun groupe de d\'ecomposition ne soit \'egal
au groupe de Galois de $L/K$.
Voici un exemple : on prend une isog\'enie cyclique $D _{0}\to C_{0}$ de courbes elliptiques sur $\C$.
On \'etend la situation \`a $\C((t))$. Pour toute extension finie $\C((s))/\C((t)$, l'application
$D(\C((s))) \to C(\C((s)))$ est surjective. Ainsi tous les groupes de d\'ecomposition sont triviaux.

Des exemples similaires se trouvent dans \cite{rimwhap}.

\section{Cohomologie  \`a coefficients $\mu_{n}$ et $\G_{m}$}

On dit qu'un corps $k$ est {\it \`a cohomologie galoisienne finie}
si pour tout module galoisien fini $M$ et tout $i \geq 0$, le 
groupe $H^i(k,M)$ est fini.

\subsection{Cohomologie \'etale des courbes}

 Soient $k=\C((t))$ et $\overline{k}=\cup_{n}\C((t^{1/n}))$.
 Soit $\mathfrak{g}=\Gal(\overline{k}/k)$.
 On a un isomorphisme naturel $\mathfrak{g}=\widehat{\Z}(1)$.

\begin{prop}\label{cohocourbes}

Soit $C/k$ une courbe projective lisse g\'eom\'etriquement connexe de 
jacobienne $J$, et soit $K=k(C)$
son corps des fonctions. Soit $\overline{C}=C\times_{k}\overline{k}$.
On note $I(C)$ l'indice de $C$ (i.e. le pgcd des degr\'es sur $k$ des points
ferm\'es de $C$) de la courbe.

\smallskip

(i) Pour tout entier $n\geq 1$, et tout entier $i \in \Z$, on a des isomorphismes canoniques
$ H^2(\overline{C},\mu_{n}^{\otimes i}) = \mu_{n}^{\otimes i-1}$
et    $ H^2(\overline{C}, \Q/\Z(i))=\Q/\Z(i-1).$

(ii) Pour tout entier $n\geq 1$, et tout entier $i \in \Z$, on a des  isomorphismes canoniques $H^3(C,\mu_{n}^{\otimes i})= \mu_{n}^{\otimes i-2}$ et  $ H^{3}(C,\Q/\Z(i)) = \Q/\Z(i-2).$

(iii) Soit $U \subset C$, $U \neq C$ un ouvert de $C$. On a $H^3(U,\G_{m})=0$
et $H^3_{c}(U,\G_{m})=H^3(C,\G_{m})=\Q/\Z(-1)$. Pour tout entier $n \geq 1$ et tout $i \in \Z$
on a $H^3_{c}(U,\mu_{n}^{\otimes i})=H^3(C,\mu_{n}^{\otimes i})=\mu_{n}^{\otimes i-2}$, et $H^3(U,\mu_n^{\otimes i})=0$. De plus 
les groupes $H^r(U,\mu_{n}^{\otimes i})$ sont finis pour tout $r \geq 0$ 
et tout $i \in \Z$, et le groupe de Brauer $\Br C$ est un groupe 
divisible de type cofini, 
quotient de $H^1(k,J)$ par un groupe cyclique d'ordre $I(C)$.

(iv) Pour chaque place $v \in \Omega$, on a $\Br(K_{v}) \oi H^1(\kappa_{v},\Q/\Z) = \Q/\Z(-1)$.

(v) On a une suite exacte naturelle
$$0 \to \Br C \to \Br K \to \bigoplus_{v \in C^{1}} \Br K_{v} \to \Q/\Z(-1) \to 0.$$

\end{prop}

\begin{proof}
Pour tout entier $n\geq 1$ on a la suite exacte de Kummer
$$ 1 \to \mu_{n} \to \G_{m} \stackrel{x\mapsto x^n}{\longrightarrow}  \G_{m} \to 1.$$
Le th\'eor\`eme de Tsen et les propri\'et\'es connues du groupe de Brauer donnent
$H^2({\overline C},\G_{m})=0$. De la suite de Kummer et de la structure du sch\'ema de Picard
 on d\'eduit
$$ \Z/n = \Pic {\overline C}/n \oi H^2(\overline{C},\mu_{n}).$$
Ceci donne (i). 

\smallskip

Comme on a ${\rm cd}(k) \leq 1$ et $H^{i}({\overline C},\mu_{n}^{\otimes i})=0$
pour $i \geq 3$, la suite spectrale
$$E_{2}^{pq}=H^p(k,H^q({\overline C}, \mu_{n}^{\otimes i})) \Longrightarrow H^n(C,\mu_{n}^{\otimes i})$$
donne
$$H^3(C,\mu_{n})=H^1(k, H^2(\overline{C},\mu_{n}^{\otimes i}))=H^1(k,\mu_{n}^{\otimes i-1})=\mu_{n}^{\otimes i-2},$$
ce qui donne (ii).

\smallskip

On a $H^{q}(\overline{C},\G_{m})=0$ pour $q \geq 2$. 
La suite spectrale
$$E_{2}^{pq}=H^p(k,H^q({\overline C}, \G_{m}))  \Longrightarrow H^n(C,\G_{m})   $$
et la  nullit\'e de $H^2(k,\Pic^0 {\overline C})$
(qui vient de ce que $\Pic^0 {\overline C}$ est divisible)
donne 
$$ H^3(C,\G_{m})= H^2(k,\Pic \overline{C}) = H^2(k,\Z)= H^1(k,\Q/\Z)= \Q/\Z(-1).$$
  La suite spectrale analogue
pour $U$ donne la finitude de $H^r(U,\mu_{n}^{\otimes i})$ via 
\cite{milneEC}, Th. VI.5.5 qui vaut sur un corps alg\'ebriquement clos.
Cette m\^eme suite spectrale donne
$H^3(U,\G_{m})=H^2(k, \Pic {\overline U}) $. Le groupe $\Pic {\overline C}$
est extension de $\Z$ par un groupe divisible. Pour $U$ ouvert strict de $C$, le groupe $\Pic{\overline U}$
est donc extension d'un groupe fini par un groupe divisible. Ceci implique $H^2(k, \Pic{\overline U})=0 $,
et donc  $H^3(U,\G_{m})=0$. De m\^eme $H^3(U,\mu_n^{\otimes i})=0$ car 
${\rm cd}(k) \leq 1$ et $H^q(\ov U,\mu_n^{\otimes i})=0$ pour $q >1$ 
vu que $\ov U$ est une courbe affine sur un corps alg\'ebriquement clos.

 Les corps $\kappa(v)$ sont
de dimension cohomologique 1, les $H^{q}(\kappa(v),\G_{m})$ sont donc nuls pour $q \geq 2$.
La proposition~3.1 (2) 
de \cite{HaSz} donne l'isomorphisme 
$H^3_{c}(U,\G_{m}) \oi H^3(C,\G_{m})$
et $H^3_{c}(U,\mu_{n}^{\otimes i})=H^3(C,\mu_{n}^{\otimes i})$ pour tout $i \in \Z$. Enfin on d\'eduit aussi de la suite spectrale et de la nullit\'e 
de $\Br {\overline C}$ (Tsen) un isomorphisme
$$H^1(k, \Pic {\overline C}) = \Br C.$$
De la suite exacte
$$ 0 \to J({\overline k}) \to \Pic {\overline C} \to \Z \to 0$$
et de ${\rm cd}(k) \leq 1$
on d\'eduit la suite exacte
$$ 0 \to \Z/I(C) \to H^1(k,J) \to H^1(k, \Pic {\overline C}) \to 0$$
car $H^1(k,\Z)=0$.
De ${\rm cd}(k) \leq 1$ et de la divisibilit\'e de $J$ on d\'eduit que
$H^1(k,J) $ est divisible, ce qui implique la divisiblit\'e de $\Br C$.
Ceci \'etablit (iii). 

\smallskip

L'\'enonc\'e (iv) est standard. Le calcul de $H^3(C,\G_{m})$ et 
la nullit\'e de $H^3(K,\G_{m})$, ainsi que
les isomorphismes $H^2(K_{v},\G_{m}) \oi H^1(\kappa_{v},\Q/\Z)$
et \cite[III, Prop. (2.1)]{GrBr} donnent l'\'enonc\'e (v). 
On peut aussi l'obtenir \`a moindres frais en prenant la cohomologie galoisienne de la suite exacte
$$ 1 \to {\overline k}(C)^{\times}/{\overline k}^{\times} \to {\rm Div} \ {\overline C} \to \Pic{\overline C}\to 0.$$
\end{proof}

\begin{rema}
Certains des r\'esultats de cette sous-section sont discut\'es dans \cite[Chap. I, Appendix A, p. 131--138]{milneADT}.
\end{rema}

\subsection{Les groupes $\cyr{X}^1(K,\mu_{n}^{\otimes i})$}

 \begin{prop}\label{Sha1Kmun}
 
 Soient  $F$ un corps de caract\'eristique nulle, $k=F((t))$,  et $C$ une $k$-courbe 
 projective, lisse, g\'eom\'etriquement int\`egre. Soit $K$ le corps 
 des fonctions de $C$. Soit
  ${\mathcal C}/F[[t]]$
 un $F[[t]]$-mod\`ele
  r\'egulier, int\`egre, projectif de $C/k$.   Soit ${\mathcal C}_{0}/F$ la fibre  sp\'eciale et
 $C_{0} =( {\mathcal C}_{0})_{\rm red}$ la fibre sp\'eciale r\'eduite de  ${\mathcal C}/F[[t]]$.

 Consid\'erons les groupes suivants.

(i) Le groupe $\cyr{X}^1(K,\mu_{n}) = \Ker [H^1(K,\mu_{n}) \to  \prod_{v \in \Omega} H^1(K_{v},\mu_{n})].$

(ii) Le noyau de la fl\`eche d'\'evaluation
$ H^1(C,\mu_{n}) \to \prod_{v \in \Omega} H^1(\kappa_{v},\mu_{n}),$
o\`u $\kappa_{v}$ d\'esigne le corps r\'esiduel en $v$.

(iii) Le groupe $H^1({\mathcal C},\mu_{n})$.

(iv) Le groupe $H^1({\mathcal C}_{0},\mu_{n})$.

(v) Le groupe $H^1(C_{0},\mu_{n})$.

Ils s'identifient tous \`a des sous-groupes de $H^1(C,\mu_{n}) \subset H^1(K,\mu_{n})$.

Les groupes (i) et (ii) co\"{\i}ncident.
Les groupes (iii), (iv), (v) co\"{\i}ncident. Le groupe (ii) est un sous-groupe du groupe (iii).

Si $F$ est alg\'ebriquement clos, tous les groupes co\"{\i}ncident, et ce sont des groupes finis.

  \end{prop}
  
  \begin{proof} 
On a la suite exacte
  $$ 0 \to H^1(C,\mu_{n}) \to H^1(K,\mu_{n}) \to \bigoplus_{v \in C^{(1)}} \Z/n,$$
  la suite exacte compatible
  $$ 0 \to H^1(O_{v},\mu_{n}) \to H^1(K_{v},\mu_{n}) \to \Z/n,$$
  et des isomorphismes
  $$H^1(O_{v},\mu_{n}) \oi  H^1(\kappa_{v},\mu_{n}).$$
 Ceci \'etablit l'isomorphisme entre (i) et (ii). 
 
Comme ${\mathcal C}$ est r\'egulier, le morphisme    $ \Spec K \hookrightarrow {\mathcal C}$
induit  un isomorphisme $\mu_{n,{\mathcal C}} \oi i_{*}\mu_{n,K}$. On a donc des plongements 
 $$ H^1({\mathcal C},\mu_{n}) \hookrightarrow H^1(C,\mu_{n})  \hookrightarrow H^1(K,\mu_{n}).$$

Soit $\xi \in H^1(C,\mu_{n})$ dans le groupe (ii). 
Soit $E \subset {\mathcal C}_{0}$   une composante
int\`egre  de la fibre sp\'eciale  ${\mathcal C}_{0}$.
 Soit $P $ un point ferm\'e  lisse de $E$ qui n'appartient \`a aucune autre composante de ${\mathcal C}_{0}$. Il existe un sous-sch\'ema r\'egulier int\`egre $Z \subset {\mathcal C}$,
fini sur $\C[[t]]$, intersectant $E$ transversalement au point $P$. Soit $v \in C^{(1)}$ le point g\'en\'erique de $Z$. Les restrictions de $C$ \`a $v$ et de $E$ \`a $P$ induisent un diagramme commutatif pour les fl\`eches 
r\'esidus~:

$$
\begin{CD}
 H^1(C,\mu_{n}) @>{\partial_{\C(E)}}>>  \Z/n \cr
 @VVV @VV{=}V \cr
 H^1(\kappa_v,\mu_{n})  @>{\partial_P}>>  \Z/n. 
 \end{CD}
 $$
 
Comme $\xi$ a une image nulle dans $H^1(\kappa_{v},\mu_{n})$, 
le  r\'esidu de $\xi$ en $E$ est nul.
Ainsi tous les r\'esidus de $\xi$ aux points de codimension 1 du sch\'ema r\'egulier ${\mathcal C}$
sont nuls, ceci implique $\xi \in H^1({\mathcal C}, \mu_{n})$. Le groupe (ii) est donc dans le groupe (iii).

 L'isomorphisme entre les groupes  (iii) et  (iv) est un cas sp\'ecial du th\'eor\`eme de changement de base propre  (\cite[ VI, Cor. 2.7]{milneEC}).
  
L'isomorphisme entre les groupes  (iv) et (v)  : la cohomologie \'etale \`a coef\-ficients $\Z/n$
 est invariante
par passage au sous-sch\'ema r\'eduit (cela vient de ce que si $X$ est un 
sch\'ema, les petits 
sites \'etales de $X$ et $X_{\rm red}$ sont isomorphes, ce qui r\'esulte
formellement de \cite{milneEC}, Th. I.3.23).

L'adh\'erence d'un point $v$ de $C$
dans ${\mathcal C}$ est un sch\'ema $Z_{v}$ int\`egre   fini sur $F[[t]]$, hens\'elien, de corps des fractions $\kappa_{v}$. Supposons $F$  alg\'ebriquement clos. On a alors  $ H^1(Z,\mu_{n})=0$.
La  fl\`eche  de restriction compos\'ee 
$H^1({\mathcal C},\mu_{n}) \to H^1(Z,\mu_{n}) \to H^1(\kappa_{v},\mu_{n})$  est donc  nulle.
Ainsi le groupe (iii) est un sous-groupe du groupe (ii). Dans ce cas
on a d\'ej\`a que $H^1(C,\mu_n)$ est fini via la suite spectrale 
de Hochschild-Serre car le corps $k$ est \`a cohomologie galoisienne
finie et $H^1(\ov C,\mu_n)$ est fini par \cite{milneEC}, Th. VI.5.5.

  \end{proof}
   
     \begin{cor}\label{corshaZ/n}
   
   Soient $k=\C((t))$, $C$ une $k$-courbe projective, lisse, g\'eom\'etriquement int\`egre
  et  ${\mathcal C}/\C[[t]]$
 un mod\`ele int\`egre r\'egulier et propre de $C/k$.
 Soit $K$ le corps des fonctions de $C$. 
 Soit ${\mathcal C}_{0}$ la fibre  sp\'eciale et
 $C_{0} $ la fibre sp\'eciale r\'eduite.
 Les groupes suivants s'identifient naturellement \`a un m\^eme sous-groupe fini  de $H^1(K,\Z/n)$.
 
(i) Le groupe $\cyr{X}^1(K,\Z/n) = \Ker [H^1(K,\Z/n) \to  \prod_{v \in \Omega} H^1(K_{v},\Z/n)].$

(ii) Le noyau de la fl\`eche d'\'evaluation
$ H^1(C,\Z/n) \to \prod_{v \in \Omega} H^1(\kappa_{v},\Z/n),$
o\`u $\kappa_{v}$ d\'esigne le corps r\'esiduel en $v$.

(iii) Le groupe $H^1({\mathcal C},\Z/n)$.

(iv) Le groupe $H^1({\mathcal C}_{0},\Z/n)$.

(v) Le groupe $H^1(C_{0},\Z/n)$.

    \end{cor}

  \begin{cor}\label{shadivisible}
   Soient $k=\C((t))$, $C$ une $k$-courbe projective, lisse, g\'eom\'etriquement int\`egre.
  Le groupe  $\cyr{X}^1(K,\Q/\Z) $
  est un groupe divisible de  type cofini.
   \end{cor}

\begin{proof}

D'apr\`es le corollaire~\ref{corshaZ/n}, le groupe $\Sha^1(K,\Q/\Z)$
est isomorphe \`a $H^1(C_{0},\Q/\Z)$, o\`u $C_{0}$ est une courbe
propre sur $\C$. Il suffit alors d'utiliser la proposition bien 
connue suivante~: 

{\it Soit $C$ une courbe propre sur un corps alg\'ebriquement 
clos de caract\'eristique z\'ero. Alors
le groupe $H^1(C,\Q/\Z)$ est un groupe divisible.}

Faute de r\'ef\'erence, nous indiquons une preuve de cette proposition. 
On peut supposer $C$ r\'eduite.
Soit $\pi : \widetilde{C} \to C$ la normalisation. On dispose 
d'une suite exacte de 
faisceaux
$$ 0 \to \Q/\Z \to \pi_{*}(\Q/\Z) \to \bigoplus_{m \in S} (i_m)_* F_{m } 
\to 0$$
o\`u $S$ est un ensemble fini de points de $C$ et o\`u chaque  $F_{m}$
est une somme directe d'exemplaires de $\Q/\Z$.
On en d\'eduit une suite exacte
$$ \bigoplus F_{m} \to H^1(C,\Q/\Z) \to H^1(\tilde{C}, \Q/\Z) \to 0.$$
La courbe $\tilde{C}$ est une union disjointe de courbes propres  et lisses,
pour une telle courbe $D$, le groupe $H^1(D,\Q/\Z)$  s'identifie au groupe
des points de torsion de la jacobienne, qui est un groupe divisible.

\end{proof}

Pour tout module galoisien $M$ et tout $i >0$, on notera 
$\Sha^i_{\omega}(K,M)$ (ou simplement $\Sha^i_{\omega}(M)$ si $K$ est sous-entendu) 
le sous-groupe de $H^i(K,M)$ des \'el\'ements
d'image nulle dans tous les 
$H^i(K_{v},M)$ sauf peut-\^etre un nombre fini.

   \begin{prop}\label{sha1sha1omega}
 Soit $k=\C((t))$. 
 
 (i) On  a
  $\cyr{X}^1(K,\mu_{n}) =\cyr{X}_{\omega}^1(K,\mu_{n}).$
  
(ii) On  a $\cyr{X}^1(K,\Z/n) =\cyr{X}_{\omega}^1(K,\Z/n).$

(iii) On  a $\cyr{X}^1(K,\Q/\Z) =\cyr{X}_{\omega}^1(K,\Q/\Z),$ et ce groupe est divisible.

(iv) Pour tout $\Gal({\overline K}/K)$-module de permutation $P$,
on a $\cyr{X}^2(K,P)= \cyr{X}^2_{\omega}(K,P)$, et ce groupe est divisible.

  \end{prop}
  
  \begin{proof}
  L'\'enonc\'e (i) implique clairement (ii) et 
  (iii), la divisibilit\'e venant du corollaire \ref{shadivisible}.
  Soit $L/K$ une extension finie de corps, 
   et $G_{L } \subset G_{K}$ les groupes de Galois absolus.
  Soit $P=\Z[G_{K}/G_{L}]$. On a $H^2(K,P)=H^2(L,\Z)$ et l'\'enonc\'e analogue
  sur tout compl\'et\'e $K_{v}$ de $K$. On a donc 
{\small $$ \cyr{X}^2_{\omega_{K}}(K,P)= \cyr{X}^2_{\omega_{L}}(L,\Z)=
   \cyr{X}^1_{\omega_{L}}(L,\Q/\Z) = \cyr{X}^1(L,\Q/\Z) = \cyr{X}^2 (K,P)$$}
par application du r\'esultat (iii) au corps $L$, qui est le corps des fonctions d'une courbe sur un corps
extension finie de $\C((t))$, donc  isomorphe \`a $\C((t))$. La divisibilit\'e du groupe consid\'er\'e suit aussi de (iii).

Il suffit donc d'\'etablir l'\'enonc\'e (i).
  
 On a $H^1(K,\mu_{n})=K^{\times}/K^{\times n}$ et 
 $H^1(K_{v},\mu_{n})=K_{v}^{\times}/K_{v}^{\times n}$. 
 Soit $f \in K^{\times}$ dont l'image dans $H^1(K,\mu_{n})$ est 
dans $\Sha^1_{\omega}(K,\mu_n)$.
Supposons d'abord $v(f)$ nul modulo $n$ pour toute place $v \in C^{(1)}$.  
Alors le rev\^etement $D:=C(^n \sqrt{f})$ de $C$ est \'etale. 
Pour toute extension finie $k'$ de $k$, l'image de l'application 
$D(k') \to C(k')$ est ferm\'ee par propret\'e de $D \to C$, et le 
compl\'ementaire de cette image dans $C(k')$ est fini car par hypoth\`ese, 
le rev\^etement $D \to C$ est totalement d\'ecompos\'e en presque 
toute place $v \in C^{(1)}$. On en d\'eduit  que $D(k') \to C(k')$ 
est surjectif, et le rev\^etement $D \to C$ \'etant \'etale, ceci 
implique qu'il est totalement d\'ecompos\'e en toute place $v \in C^{(1)}$,
i.e. la classe de $f$ est dans $\Sha^1(K,\mu_n)$.

\smallskip

Supposons d\'esormais qu'il existe $v \in C^{(1)}$ 
avec $v(f)$ non nul modulo $n$. 
   Soit $P$ le point ferm\'e de $C$ d\'efini par $v$ et soit $L=\kappa_{v}$ son corps r\'esiduel.
   
     Soit $\pi \in K^{\times}$ une uniformisante de $C_{L}$ en $P$.
   On a donc $f=u.\pi^r$ avec $r \neq 0 \ {\rm mod} \ n$ 
et $u\in K^{\times}$ unit\'e en $P$.
   
   Par le th\'eor\`eme des fonctions implicites sur le corps $L$, 
   il existe un voisinage ouvert
 analytique $U \subset C(L)$ de $P$ et un ouvert $V \subset L$ contenant $0$
 tel que $\pi$ induise un isomorphisme entre $U$ et $V$.
 Pour $U$ suffisamment petit, et $Q \in U$, $Q \neq P$, la classe de
 $u(Q)$ dans $L^{\times}/L^{\times n} \oi \Z/n$ est constante, et la classe de $\pi(Q)$
 prend toutes les valeurs dans  $L^{\times}/L^{\times n} \oi \Z/n$. Il existe donc
 une infinit\'e de points  $Q \in C(L)$, 
 de valuation associ\'ee $v=v_{Q}$,  tels que $v(f(Q)) \neq 0 \ {\rm mod} \ n$, et donc
 $f \notin K_{v}^{\times}/K_{v}^{\times n}$.
Pour conclure  que toute classe dans $\cyr{X}_{\omega}^1(K,\mu_{n})$
 appartient \`a $H^1(C,\mu_{n})$, il suffit de montrer qu'on peut 
trouver une infinit\'e de points $Q$ comme ci-dessus qui de plus 
d\'efinissent des points ferm\'es de $C$. 
Or, pour $U$ suffisamment petit, $U\setminus P$ est une union finie 
de $n$ ouverts non vides
$U_{i}, i=1,\dots,n$ tels que sur $U_{i}$ on ait $v_{L}(\pi(Q))=i \in \Z/n$.
On conclut avec le lemme suivant~:

\begin{lem}
Soient $K$ un corps complet pour une valuation discr\`ete 
et $L/K$ une extension finie s\'eparable.
Soit $X$ une $K$-vari\'et\'e lisse int\`egre.
Dans $X(L)$, la r\'eunion des $X(F) \subset X(L)$ pour $F/K$ 
sous-extension propre de $L/K$
est un ferm\'e dont le compl\'ementaire est dense dans $X(L)$.
\end{lem}

Le lemme se montre facilement en se ramenant (via le th\'eor\`eme 
des fonctions implicites) au cas o\`u $X$ est l'espace affine, en 
prenant un ouvert de Zariski non vide de $X$ \'equip\'e d'un 
morphisme \'etale vers ${\bf A}_K^n$.

\end{proof}

  \begin{remas}
 
 (1) Soit $k=\C((t))$ et $C$ une $k$-courbe projective, lisse, g\'eom\'etriquement int\`egre d'indice $I=I(C)>1$.
 Alors l'application 
 compos\'ee
 $$ H^1(k,\Z/I) \to H^1(C,\Z/I) \to \prod_{v \in \Omega} H^1(\kappa_{v},\Z/I)$$
 est nulle, et l'application  compos\'ee $$H^1(k,\Z/I) \to H^1(C,\Z/I)  \to H^1(K,\Z/I)$$
 est injective. On a donc dans ce cas $$\Z/I \simeq H^1(k,\Z/I) \subset \cyr{X}^1(K, \Z/I) \subset \cyr{X}^1(K, \Q/\Z).$$
 
 \medskip
 
 (2) Un exemple concret est fourni par la courbe  $C$ d'\'equation $$X^3+tY^3+t^2Z^3=0.$$
 On a $I(C)=3$. La classe de $t \in k^{\times}/k^{\times 3}$ s'annule 
dans tout $ \kappa_{v}^{\times}/\kappa_{v}^{\times 3}$, c'est un 
\'el\'ement non nul de $\Sha^1(\Z/3)$. Le 
mod\`ele r\'egulier minimal de $C$ 
a une fibre sp\'eciale de type $3I_{0}$. La courbe $C_{0}$    \'etant une courbe elliptique (lisse) sur $\C$, le corollaire \ref{corshaZ/n} donne donc  $\cyr{X}^1(K,\Z/3)=(\Z/3)^2$.

 \medskip
 
  (3)  Comme dans  \cite{CTPaSu}, on peut consid\'erer le sous-groupe 
   $$\cyr{X}^1_{\rm total}(K,\Z/n) = \Ker [H^1(K,\Z/n) \to \prod_{v \in 
{\mathcal C}^{(1)} } H^1(K_{v},\Z/n)]$$
   de $\cyr{X}^1(K,\Z/n)$.
    Il s'identifie au groupe des \'el\'ements
  de  $H^1(C_{0}, \Z/n)$ qui s'annulent 
  au point g\'en\'erique de chaque composante connexe de $C_{0}$.
  On peut aussi v\'erifier qu'il s'identifie au groupe des \'el\'ements de $H^1(K,\Z/n) $
  dont la restriction \`a tout compl\'et\'e de $K$ en une valuation discr\`ete de rang 1  quelconque de $K$ est nulle.

  Le groupe $\cyr{X}^1_{\rm total}(K,\Z/n)$ est  nul si la fibre
  sp\'eciale r\'eduite $C_{0}$ est lisse, car $H^1(C_{0},\Z/n) \to H^1(\C(C_{0}),\Z/n)$
  est alors  injectif. Ceci s'applique \`a l'exemple (2) ci-dessus.
 
  Si chaque composante de $C_{0}$ est une courbe lisse de genre z\'ero,
  les applications de restriction de $H^1(C_{0},\Z/n)$ 
  au point g\'en\'erique de chaque composante
  se factorisent  par $H^1(\P^1_{\C},\Z/n)=0$, on a  alors $\cyr{X}^1_{\rm total}(K,\Z/n) = H^1(C_{0},\Z/n)$.
  Un exemple avec $\cyr{X}^1_{\rm total}(K,\Z/n) $ non nul est donc fourni
  par toute courbe elliptique $C/k$ dont la fibre sp\'eciale est de type $I_{m} $ avec $m\geq 2$.
  Comparer avec les paragraphes 7  et 8  (et particuli\`erement le corollaire   8.1) de
   \cite{HHK2}. 
  
\medskip

(4) Si $K$ est un corps de nombres, le th\'eor\`eme de Tchebotarev donne 
$\Sha^1(K,\Z/n)=\Sha^1_{\omega}(K,\Z/n)=0$, et de m\^eme avec $\Q/\Z$ au lieu 
de $\Z/n$. On a donc aussi dans ce cas $\Sha^2_{\omega}(K,P)=\Sha^2(K,P)=0$
pour tout module de permutation $P$. Par contre $\Sha^1_{\omega}(K,\mu_n)$ 
peut contenir strictement $\Sha^1(K,\mu_n)$, c'est par exemple le 
cas pour $K=\Q$ et $n=8$ (Grunwald-Wang).

\end{remas}

\subsection{Le groupe $\cyr{X}^2(K,\mu_{n})$}

\begin{prop}\label{Sha2Kmun}
 Soient $k=\C((t))$, $C$ une $k$-courbe projective, lisse, g\'eom\'etriquement int\`egre
  et  ${\mathcal C}/\C[[t]]$
 un mod\`ele int\`egre r\'egulier et propre de $C/k$.
 Soit $K$ le corps des fonctions de $C$. 
 Soit ${\mathcal C}_{0}$ la fibre  sp\'eciale et
 $C_{0} $ la fibre sp\'eciale r\'eduite.
 Les groupes suivants s'identifient naturellement \`a un m\^eme sous-groupe fini  de $H^2(K,\mu_{n})$.

(i) Le groupe $\cyr{X}^2(K,\mu_{n}) = \Ker [H^2(K,\mu_{n}) \to  \prod_{v \in \Omega} H^2(K_{v},\mu_{n})].$

(ii) Le groupe $_n \Br(C)$.

(iii) Le groupe $H^3_{C_{0}}({\mathcal C}, \mu_{n})$.

(iv) Le   groupe $\Hom(H^1(C_{0}, \mu_{n}), \Z/n)$.

(v) Le groupe  $\Hom(\cyr{X}^1(K,\mu_{n}),\Z/n)$.
\end{prop}

 \begin{proof} On a $H^2(K,\mu_{n})=_n \Br K $ et 
$H^2(K_{v},\mu_{n})=_n \Br K_{v} $;
 la proposition \ref{cohocourbes}(v) montre que les groupes (i) et (ii) co\"incident.

On a  $\Br {\mathcal C} \oi \Br C_{0}=0$ (Artin, Grothendieck \cite[III, Thm. 3.1 et Cor. 5.8]{GrBr}).
Comme $ {\mathcal C}$ est r\'egulier, la
 restriction $\Pic  {\mathcal C} \to \Pic C$
est surjective. 
De la suite exacte de localisation pour la cohomologie \`a support
$$ H^2({\mathcal C}, \mu_{n}) \to H^2(C, \mu_{n}) \to H^3_{C_{0}}({\mathcal C}, \mu_{n}) \to H^3({\mathcal C}, \mu_{n})  $$
et de la suite de Kummer pour ${\mathcal C}$ et pour $C$
on d\'eduit donc une suite exacte
$$0 \to _n \Br (C)  \to H^3_{C_{0}}({\mathcal C}, \mu_{n}) \to H^3({\mathcal C}, \mu_{n}).$$
Le th\'eor\`eme de changement de base propre donne
$$ H^3({\mathcal C}, \mu_{n}) \oi  H^3(C_{0}, \mu_{n}) = 0.$$ On a donc 
$_n \Br (C)  \oi H^3_{C_{0}}({\mathcal C}, \mu_{n}) $.
 L'isomorphisme entre les groupes (iii) et (iv) est un cas particulier d'un th\'eor\`eme g\'en\'eral de dualit\'e 
 qu'on peut trouver explicit\'e dans \cite[Prop. 2.6]{EW}.
 
 L'identification entre le groupe (iv) et le groupe (v)  r\'esulte alors
 de la proposition \ref{Sha1Kmun}.
 \end{proof}

\begin{ex}\label{contreexrationnel}

D'apr\`es les propositions \ref{cohocourbes} et \ref{Sha2Kmun}, on
a les \'egalit\'es 
${}_{n}\Br C  = \Ker [{}_{n}\Br K   \to \prod_{v} {}_{n}\Br K_{v}]$ et ${}_{n}\Br(C)= \Hom(H^1(C_{0}, \mu_{n}), \Z/n)$.
  Supposons $H^1(C_{0}, \mu_{n})\neq 0$. Soit $\alpha \neq 0$ dans
$ \Ker [{}_{n}\Br K  \to \prod_{v} {}_{n}\Br K_{v}]$. Il existe une extension finie $L/K$ de corps,
avec $\alpha_{L}=0 \in \Br L$, extension qu'on peut prendre de degr\'e divisant $n$, 
car pour toute alg\`ebre simple centrale $D$ sur $K$, indice et exposant de $D$ co\"{\i}ncident, 
comme rappel\'e au \S \ref{lescorps}.
 Le plongement naturel $\G_{m,K} \to R_{L/K}\G_{m,L}$ induit une suite exacte de $K$-tores alg\'ebriques
$$ 1 \to \G_{m,K} \to R_{L/K}\G_{m,L} \to T \to 1.$$
Les suites de cohomologie associ\'ees sur $K$ et sur les $K_{v}$, le th\'eor\`eme de Hilbert 90 et le lemme de Shapiro
donnent une suite exacte
$$0 \to \cyr{X}^1(K,T) \to  \cyr{X}^2(K,\G_{m}) \to \cyr{X}^2(L,\G_{m}).$$
On a donc $ \cyr{X}^1(K,T) \neq 0$. On notera que le $K$-tore $T$
est $K$-birationnel \`a l'espace projectif. La situation ici est donc diff\'erente
de celle qui vaut sur les corps globaux (\cite{sansuc}, Cor.~9.7) 
ou sur les corps de fonctions d'une variable sur 
les corps $p$-adiques (\cite{HaSz}, Cor.~5.7).)

Il est facile de donner un exemple num\'erique.  
On part de la courbe elliptique constante $E$ donn\'ee par l'\'equation affine
$$y^2=x(x-1)(x+1),$$ et on note $C=E\times_{\C}\C((t))$.
 La classe de quaternions $\alpha=(x,t)$ non ramifi\'ee sur $C$ est nulle dans 
chaque $K_{v}$ mais pas dans $K$.
On peut prendre alors pour $T$ soit $X^2-xY^2=1$ (avec l'espace 
principal homog\`ene $X^2-xY^2=t$), soit 
$X^2-tY^2=1$ (avec l'espace principal homog\`ene $X^2-tY^2=x$).

\end{ex}

\section{Dualit\'e locale pour la cohomologie des groupes 
finis et des tores}

 Dans ce paragraphe, on note $K=\C((t))((u))$. 
 Toute extension finie $L$ de $K$ est de la forme $\C((t'))((u'))$.

\begin{lem} \label{compatkummer}
Soit $S$ un sch\'ema. Soit $\calt$ un tore sur $S$ et $\widehat{{\calt}}$
son groupe des caract\`eres. Alors on a un diagramme commutatif 
(dans la cat\'egorie d\'eriv\'ee born\'ee des faisceaux fppf sur 
$S$)~:

$$
\begin{CD}
\calt \otimes \widehat{\calt} @>>> \G_m \cr
@VVV @VVV \cr
(_n \calt)[1] \otimes \widehat{\calt}/n @>>> \mu_n[1],
\end{CD}
$$

\smallskip

\noindent o\`u la fl\`eche $\G_m \to \mu_n[1]$ est induite par 
le triangle exact
$$\mu_n \to \G_m \stackrel{.n}{\to} \G_m \to \mu_n[1].$$
\end{lem}

\begin{proof} 
On part du triangle exact 
$$_n \calt \to \calt \stackrel{.n}{\to} \calt \to (_n \calt)[1],$$
que l'on peut tensoriser par le faisceau $\widehat{\calt}$ 
(qui est localement isomorphe 
\`a $\Z^r$ pour la topologie \'etale, avec $r \geq 0$). Comme 
$\widehat{\calt}=Hom(\calt, \G_m)$, on en d\'eduit le r\'esultat 
en observant que $$(_n \calt)[1] \otimes \widehat{\calt}=(_n \calt)[1] \otimes
\widehat{\calt}/n$$
et $\widehat{\calt}/n=Hom(_n \calt, \G_m)=Hom(_n \calt, \mu_n)$.
\end{proof}

\begin{lem}
(i) Pour tout $L$-groupe ab\'elien fini $M$, et tout entier $i \geq 0$,
le groupe $H^{i}(L,M)$ est fini, et nul pour $i \geq 3$.

 (ii) Pour tout $L$-groupe de type multiplicatif $M$ sur $L$,
le groupe $H^1(L,M)$ est fini. 
\end{lem}
\begin{proof}
(i) Soit $L'/L$ une extension finie galoisienne d\'eployant $M$.
La finitude r\'esulte via la suite spectrale de Hochschild-Serre 
de la nullit\'e de $H^{i}(L',\mu_{n})$ pour $i\geq 3$ et de la finitude des
 groupes $H^1(L',\mu_{n})$ ainsi que $H^2(L',\mu_{n})= _n \Br L' $.
\end{proof}

\begin{prop}\label{dualitelocalefinie}
Soit $F$ un module galoisien fini, de dual de Cartier
$F^0=Hom(F,\Q/\Z(1))$. Pour $0 \leq i \leq 2$,
on  a des dualit\'es parfaites de groupes finis
$$H^{i}(K  , F) \times H^{2-i}(K , F^0) \to H^2(K ,\Q/\Z(1))=\Q/\Z(-1).$$
\end{prop}

\begin{proof}
 \cite[\S 5.1, Exercice 2]{CG}.
\end{proof}

Pour $T$ un $K$-tore, et $\widehat{T}=Hom(T,\G_{m,K})$, on d\'efinit un accouplement
$$H^{i}(K, T) \times H^{2-i}(K, \widehat{T}) \to  H^2(K,\G_{m})= \Q/\Z(-1).$$

\begin{prop}\label{dualitelocaletore}
Pour $i=1$, l'accouplement ci-dessus est un accouplement parfait de 
groupes finis. Pour $i=0$, il induit une dualit\'e parfaite entre  le compl\'et\'e profini
de $H^0(K,T)$ et le groupe discret 
$H^2(K,\widehat T)$, et pour $i=2$ une dualit\'e parfaite entre le 
groupe discret $H^2(K,T)$ et le compl\'et\'e profini de
$H^0(K,\widehat T)$.
\end{prop}

\begin{proof}
Notons $A^{D}:=\Hom(A,\Q/\Z(-1))$.
 Pour tout $n >0$, on a un diagramme commutatif 
(via le lemme~\ref{compatkummer}) de groupes finis, 
\`a lignes exactes 
{\small
$$
\begin{CD}
0 @>>> H^0(K ,T)/n @>>> H^1(K ,  _n T) @>>> _n H^1(K ,T) @>>> 0 \cr
&& @VVV @VVV @VVV \cr
0 @>>> (_n H^2(K ,\widehat T))^D @>>> H^1(K , \widehat T/n)^D @>>> 
(H^1(K ,\widehat T)/n)^D @>>> 0
\end{CD}
$$
}
La fl\`eche verticale du milieu est un isomorphisme d'apr\`es la  Proposition \ref{dualitelocalefinie}.
 Comme $H^1(K ,T)$ et $H^1(K ,\widehat T)$ 
sont finis (ils sont de type cofini, et d'exposant fini via Hilbert 90
et $H^1(K ,\Z)=0$), on en d\'eduit en prenant $n$ convenable une 
fl\`eche surjective $H^1(K ,T) \to H^1(K ,\widehat T)^D$. On \'ecrit 
alors un autre diagramme commutatif de groupes finis, 
\`a lignes exactes
$$
{\small
\begin{CD}
0 @>>> H^1(K ,T)/n @>>> H^2(K , _n T) @>>> _n H^2(K ,T) @>>> 0 \cr
&& @VVV @VVV @VVV \cr
0 @>>> (_n H^1(K ,\widehat T))^D @>>> H^0(K ,\widehat T/n)^D @>>>
(H^0(K ,\widehat T)/n)^D @>>> 0
\end{CD}
}
$$
  dont la fl\`eche verticale du milieu est  un isomorphisme
 d'apr\`es la  Proposition \ref{dualitelocalefinie}. 
En prenant encore $n$ convenable, on en d\'eduit une fl\`eche 
injective $H^1(K ,T) \to H^1(K ,\widehat T)^D$, ce qui par cardinalit\'e 
donne finalement que l'accouplement entre $H^1(K ,T)$ est 
$H^1(K ,\widehat T)$ est un accouplement parfait de groupes finis.

\smallskip

Les deux diagrammes donnent alors que les deux fl\`eches $T(K )/n \to 
(_n H^2(K ,\widehat T))^D$ et $H^2(K ,T)_n \to (H^0(K ,\widehat T)/n) ^D$
sont \'egalement des isomorphismes, ce qui ach\`eve la preuve de la 
proposition en passant \`a la limite projective (resp. inductive) sur $n$.
\end{proof}

\begin{rema} \label{topolocale}
{\rm
L'application naturelle de $T(K) $  \'equip\'e de la topologie induite par celle de $K$,
vers le compl\'et\'e profini de $T(K)$
(qui est aussi la limite projective des $T(K)/n$ pour $n>0$ entier vu que 
chaque $T(K)/n$ est fini)
est continue. L'application qui va de
$T(K)$ vers le dual 
du groupe discret $H^2(K,\widehat T)$ est donc aussi continue.
}
\end{rema}

\medskip
De la proposition ci-dessus on d\'eduit l'\'enonc\'e suivant, analogue d'un r\'esultat 
connu sur les corps $p$-adiques.
\begin{cor}
Soit $K =\C((u))((t))$. Pour tout $K$-tore coflasque $Q$, on a $H^1(K ,Q)=0$.
En  d'autre termes, tout espace principal homog\`ene sous un tel $K $-tore
poss\`ede un point rationnel.
\end{cor}

\begin{proof} 
En effet par d\'efinition, 
un $K$-tore est coflasque s'il satisfait l'\'egalit\'e 
$H^1(L,\widehat{Q})=0$ pour toute extension finie de corps $L/K$.
\end{proof}

\section{Cohomologie des courbes : localisation et finitude}

On commence par quelques rappels g\'en\'eraux.

\begin{prop} \label{injectprop}
Soit $A$ un anneau semilocal r\'egulier int\`egre de corps des 
fractions $K$.
Pour tout $A$-groupe de type multiplicatif lisse $G$, et tout entier $i=0,1,2$,
les applications de restriction $H^{i}(A,G) \to H^{i}(K,G_{K})$ sont 
injectives.
\end{prop}

\begin{proof}
C'est la proposition 2.2 page 450 de \cite{ctcras}, laquelle 
vaut plus g\'en\'eralement pour un anneau semilocal g\'eom\'etriquement 
localement factoriel.
\end{proof}

\begin{prop}\label{extension}
Soit $X$ un sch\'ema de Dedekind int\`egre de corps des fractions $K$.
Soit $U \subset X$ un ouvert non vide. 

(i) Soit $M$ un $X$-sch\'ema en groupes finis ab\'eliens
\'etale d'ordre premier aux caract\'eristiques r\'esiduelles de $X$. 
Si une classe dans $H^{i}(U,M)$ a une image nulle dans tous les $H^{i}(K_{v},M)$
pour $v \notin U$, alors elle est dans l'image de  la restriction $H^{i}(X,M) \to H^{i}(U,M)$.

(ii) Soit $M$ un $X$-sch\'ema en groupes r\'eductifs. Si une classe 
de $H^{1}(U,M)$ a une image nulle dans tous les $H^{1}(K_{v},M)$
pour $v \notin U$, alors elle est dans l'image de la restriction 
$H^{1} (X,M) \to H^{1}(U,M)$.
\end{prop}

\begin{proof}
Pour (i), il suffit de consid\'erer la suite de localisation pour la
cohomologie \'etale de $X$, qui est fonctorielle, et le th\'eor\`eme de
puret\'e pour la cohomologie sur un anneau de valuation discr\`ete
(\cite{santa}, section~3.3).
On peut ici utiliser 
soit les hens\'elis\'es soit les compl\'et\'es.
Pour (ii), voir \cite[Lemma 4.1.3]{Harder} ou \cite[Cor. 4.8]{GiPi}.
\end{proof}

Dans la suite de ce paragraphe, on consid\`ere 
$C$ une courbe int\`egre, projective et lisse sur un corps de 
caract\'eristique z\'ero $k$,
et $U \subset C$ un ouvert de Zariski non vide. On note $K=k(C)$.

Comme $C$ est une courbe r\'eguli\`ere, pour tout faisceau \'etale $\calf$ 
sur $U$ (dont on note $F$ la restriction \`a la fibre g\'en\'erique), 
on a pour $i \geq 0$ de longues suites exactes \cite[II, Prop. 2.3]{milneADT}~:

\begin{equation} \label{longlocal}
\dots \to H^i_c(U,\calf) \to H^i(U,\calf) \to \bigoplus_{v \not \in U} H^i(K_v ^h,F) \to H^{i+1}_c(U,\calf) \dots
\end{equation}
o\`u $K_v^h=\fraco{\calo_v ^h}$ est le corps des fractions 
du hens\'elis\'e $\calo_v^h$ de l'anneau local de $C$ en $v$ (pour simplifier
les notations, on notera souvent $H^i(K_v ^h,F)$ pour $H^i(K_v^h,F_v)$, o\`u
$F_v$ est la restriction de $F$ \`a $K_v ^h$).
Si $i \geq 1$ et 
$F$ est repr\'esent\'e par un $K$-sch\'ema en groupes localement de type fini 
(par exemple un tore, un $K$-groupe fini, ou le module des caract\`eres 
d'un tore), 
on peut remplacer $K_v^h$ par le compl\'et\'e $K_v$ via le th\'eor\`eme
d'approximation de Greenberg (cf. \cite{dhsza1}, Lemme~2.7
pour un argument d\'etaill\'e). Les propri\'et\'es de fonctorialit\'e 
de la cohomologie \`a support compact sont r\'esum\'ees dans la 
proposition suivante, dont certains points sont d\'ej\`a discut\'es 
dans la Prop. 4.2. de \cite{HaSz}~:

\begin{prop} \label{compatcoh}
Avec les notations ci-dessus, on consid\`ere deux ouverts de Zariski 
non vides $U$ et $V$ de $C$ avec $V \subset U$. Soit $\calf$ 
un faisceau \'etale sur $U$. Soit $i \geq 0$ un entier. 

\smallskip

a) L'inclusion $V \subset U$ induit une fl\`eche 
$H^i_c(V,\calf) \to H^i_c(U,\calf)$ telles que si 
$W \subset V$, les fl\`eches $$H^i_c(W,\calf) \to 
H^i_c(V,\calf) \to H^i_c(U,\calf)$$
soient compatibles.

\smallskip

b) Le diagramme 
$$
\begin{CD}
H^i(U,\calf) @>>> H^i(V,\calf) \cr
@VVV @VVV \cr
\bigoplus_{v \not \in U} H^i(K_v ^h,F) @<<< \bigoplus_{v \not \in V} 
H^i(K_v ^h,F)
\end{CD}
$$
est commutatif, o\`u la fl\`eche horizontale du bas est la projection. 

\smallskip

c) Le diagramme  
$$
\begin{CD}
\bigoplus_{v \not \in U}  H^i(K_v ^h,F) @>>> \bigoplus_{v \not \in V}    
H^i(K_v ^h,F) \cr
@VVV @VVV \cr
H^{i+1}_c(U,\calf) @<<< H^{i+1}_c(V,\calf)
\end{CD}
$$
est commutatif, o\`u la fl\`eche horizontale du haut est l'injection 
canonique $(f_v) \mapsto ((f_v),0,0,...,0)$.

\smallskip

d) Le diagramme
$$
\begin{CD}
H^i_c(V,\calf) @>>> H^i_c(U,\calf) \cr
@VVV @VVV \cr
H^i(V,\calf) @<<< H^i(U,\calf) 
\end{CD}
$$
est commutatif. En particulier si on pose 
$$\sD^i(U,\calf):={\rm Im} \, [H^i_c(U,\calf) \to H^i(K,F)],$$
on a $\sD^i(V,\calf) \subset \sD^i(U,\calf)$.

\smallskip

e) Pour $i \geq 1$, on a un complexe 
\begin{equation} \label{complexf}
\bigoplus_{v \in C^{(1)}} H^{i-1}(K_v ^h,F) \to H^i_c(U,\calf) \to H^i(K,F)
\end{equation}
qui est de plus exact si pour toute place $v \in U$, la fl\`eche 
$H^i(\calo_v ^h,\calf) \to H^i(K_v ^h,F)$ est injective, par exemple si 
$\calf$ est repr\'esent\'e par un sch\'ema en groupes fini \'etale, ou 
pour $i \leq 2$ si $\calf$ est repr\'esent\'e par un sch\'ema en groupes 
de type multiplicatif.

\smallskip 

f) Soit $\calf$ un groupe de type multiplicatif sur $U$ de dual 
$\calf'=\Hom(\calf,\G_m)$. 
On note $F$ et $F'$ les fibres g\'en\'eriques 
sur $\spec K$ de $F$ et $F'$.
Soit $v \in C^{(1)}$. 
Alors pour tous entiers $i,j \geq 0$,
on a un diagramme commutatif~: 

$$
\begin{CD}
H^i(U,\calf') && \times H^{j+1}_c(U,\calf) @>>> H^{i+j+1}_c(U,\G_m) \cr
@VVV @AAA @AAA \cr
H^i(K_v ^h,F') && \times  
H^j_c(K_v^h,F) @>>> H^{i+j}(K_v^h,\G_m)
\end{CD}
$$
et de m\^eme avec $\calf$ et $\calf'$ \'echang\'es.

\end{prop}

Noter que dans ces \'enonc\'es, on pourra souvent remplacer les 
hens\'elis\'es $K_v^h$ par les compl\'et\'es $K_v$, comme on l'a 
expliqu\'e plus haut.

\begin{proof}
a) r\'esulte de la d\'efinition de $H^i_c$ et du fait que les immersions 
ouvertes $W \subset V \subset U$ induisent des morphismes compatibles
de faisceaux $(j_W)_! \calf \to (j_V)_! \calf \to (j_U)_! \calf$.

\smallskip

b) vient de ce que pour toute place $v$, la fl\`eche $H^i(U,\calf) \to 
H^i(K_v^h,F)$ est la compos\'ee des restrictions $H^i(U,\calf) \to 
H^i(K,F)$ et $H^i(K,F) \to H^i(K_v^h,F)$. 

\smallskip

Pour d\'emontrer c), il suffit (par lin\'earit\'e) de prouver que 
si $v \in C^{(1)}$ n'est pas dans $U$ (et donc pas non plus dans $V$),
le triangle
$$
\begin{diagram}
\node{} \node{H^{i+1}_c(V,\calf)}\arrow{s,r}{} \\
\node{H^i(K_v ^h,F)}\arrow{e}
\arrow{ne} \node{H^{i+1}_c(U,\calf)}
\end{diagram}
$$
est commutatif. Par (\cite{milneADT}, Prop.~1.1), on a 
$H^i(K_v^h,F)=H^{i+1}_v(\calo_v^h,(j_v)_{!} F_v)$, o\`u $F_v=F_{\mid K_v ^h}$ 
et $j_v : \spec K_v^h \to \spec \calo_v ^h$ est l'inclusion. 
Par ailleurs, la fl\`eche $H^i(K_v ^h,F) \to H^{i+1}_c(U,\calf)$ est induite
(cf. \cite{milneADT}, Lemme~2.4)
par la fl\`eche $H^{i+1}_v(C, (j_U)_! \calf) \to 
H^{i+1}(C, (j_U)_! \calf)=H^{i+1}_c(U,
\calf)$ (qui vient de la suite de localisation) et par l'identification de 
$H^{i+1}_v(C, (j_U)_! \calf)$ avec $H^{i+1}_v(\calo_v^h,(j_v)_{!} F_v)$ 
(excision, \cite{milneEC}, Cor. III.1.28).
La m\^eme chose est valable pour $V$, et on conclut alors la preuve de c) 
via la commutativit\'e du diagramme
$$
\begin{CD}
H^{i+1}_v(C, (j_V)_! \calf_{\mid V}) @>>> H^{i+1}(C, (j_V)_! \calf_{\mid V}) \cr
@VV{\simeq}V @VVV \cr
H^{i+1}_v(C, (j_U)_! \calf) @>>> H^{i+1}(C, (j_U)_! \calf). \cr
\end{CD}
$$

\smallskip

d) est d\'emontr\'e dans \cite{HaSz}, Prop.~3.1. (3). 

\smallskip 

La preuve de e) est identique \`a celle de la Prop.~4.2. de \cite{HaSz}, 
en rappelant que pour toute place $v \in C^{(1)}$, la fl\`eche 
$H^{i-1}(K_v^h,F) \to H^i_c(U,\calf)$ est d\'efinie en choisissant 
un ouvert non vide $V \subset U$ tel que $v \not \in V$, puis en 
composant la fl\`eche $H^{i-1}(K_v^h,F) \to H^i_c(V,\calf)$ 
(qui vient de la suite exacte (\ref{longlocal})) 
avec la fl\`eche $H^i_c(V,\calf) \to 
H^i_c(U,\calf)$ provenant de l'inclusion $V \subset U$; cette 
d\'efinition ne d\'epend pas de $V$ d'apr\`es c).

Il reste juste \`a v\'erifier que dans les deux cas particuliers
consid\'er\'es, la fl\`eche $H^i(\calo_v^h,\calf) \to H^i(K_v,F)$ est 
bien injective.  Le cas $i \leq 2$ avec $\calf$ de type 
multiplicatif r\'esulte de la 
Prop.~\ref{injectprop}. Dans le cas o\`u $\calf$ est fini \'etale avec 
$i$ quelconque,  
on a $H^i(\calo_v^h,\calf)=H^i(\kappa(v),F)$.
Soit $G_v$ le 
groupe de Galois de $K_v$, $I_v$ le groupe d'inertie et $G(v)$ 
le groupe de Galois du corps r\'esiduel 
$\kappa(v)$. L'injectivit\'e voulue r\'esulte alors
de ce que la suite exacte de groupes 
$$1 \to I_v \to G_v \to G(v) \to 1$$ 
est scind\'ee car $k$ est de caract\'eristique z\'ero.

\smallskip 

Pour d\'emontrer f), on note d'abord qu'il suffit de traiter le 
cas o\`u $v \not \in U$; en effet, si $v \in U$, on choisit un 
ouvert non vide $V \subset U$ tel que $v \not \in V$ et on utilise le 
diagramme (dont le rectangle sup\'erieur est commutatif)
$$
\begin{CD}
H^i(U,\calf') && \times H^{j+1}_c(U,\calf) @>>> H^{i+j+1}_c(U,\G_m) \cr
@VVV @AAA @AAA \cr
H^i(V,\calf') && \times H^{j+1}_c(V,\calf) @>>> H^{i+j+1}_c(V,\G_m) \cr
@VVV @AAA @AAA \cr
H^i(K_v ^h,F') && \times
H^j_c(K_v^h,F) @>>> H^{i+j}(K_v^h,\G_m)
\end{CD}
$$
pour se ramener (via c)) \`a v\'erifier la commutativit\'e du rectangle 
inf\'erieur, donc \`a montrer f) quand $v \not \in U$.
Comme les accouplements d'Artin-Verdier 
$$H^i(U,\calf') \times H^{j+1}_c(U,\calf) \to H^{i+j+1}_c(U,\G_m)$$
sont induits par les accouplements 
$$\ext^i_U(\calf,\G_m) \times H^{j+1}_c(U,\calf) \to H^{i+j+1}_c(U,\G_m)$$
et de m\^eme pour les accouplements locaux, il suffit de v\'erifier 
que le diagramme
$$
\begin{CD}
{\rm Hom}_U(\calf,\G_m[i]) && \times H^{j+1}_c(U,\calf) 
@>>>H^{i+j+1}_c(U,\G_m) \cr
@VVV @AAA @AAA \cr
{\rm Hom}_{K_v ^h}(F_v,\G_m[i]) && \times H^j(K_v^h,F_v) @>>>H^{i+j}(K_v^h,\G_m)
\end{CD}
$$ 
commute, les ${\rm Hom}$ \'etant pris dans la cat\'egorie d\'eriv\'ee. 
Soit alors $\alpha_U \in {\rm Hom}_U(\calf,\G_m[i])$, de restriction 
$\alpha_v \in {\rm Hom}_{K_v^h}(F_v,\G_m[i])$. Comme on l'a vu dans la 
preuve de c), on a $H^j(K_v^h,F)=H^{j+1}_v(\calo_v^h,(j_v)_! F_v)$. On conclut 
avec la commutativit\'e du diagramme 
$$
\begin{CD}
H^{j+1}_v(\calo_v^h,(j_v)_! F_v) @<<{\simeq}< H^{j+1}_v(C, (j_U)_! \calf) @>>> 
H^{j+1}(C,(j_U)_! \calf) \cr
@VV{(\alpha_v)_*}V @VV{(\alpha_U)_*}V @VV{(\alpha_U)_*}V \cr
H^{i+j+1}_v(\calo_v^h,(j_v)_! (\G_m)_v) @<<{\simeq}< H^{i+j+1}_v(C, 
(j_U)_! \G_m)
@>>> H^{i+j+1}(C,(j_U)_! \G_m)
\end{CD}
$$
et la d\'efinition de la fl\`eche 
$H^j(K_v^h,F) \to H^{j+1}_c(U,\calf)=H^{j+1}(C,(j_U)_! \calf)$, rappel\'ee
dans la preuve de c) (les isomorphismes indiqu\'es sur le diagramme 
sont obtenus via l'excision). La preuve pour $\calf$ et $\calf'$ \'echang\'es 
est identique.

\end{proof}

\begin{lem}\label{finitudeUfini}
Soit $k$ un corps \`a cohomologie galoisienne finie. Soient 
$C$ une $k$-courbe projective, lisse, g\'eom\'etriquement int\`egre, et 
$K=k(C)$.
Soit $U \subset C$ un ouvert non vide.
Pour tout $U$-sch\'ema en groupes ab\'eliens  $\calf$ fini \'etale sur $U$,
et tout entier $i \geq 0$, les groupes $H^{i}(U,\calf)$  et  $H^{i}_{c}(U,\calf)$
sont finis, et ils sont de plus nuls pour $i \geq 4$ si ${\rm cd}(k) \leq 1$.
\end{lem}

\begin{proof}  Pour tout $\Z/n$-faisceau localement constant
 $\calf$ sur toute $k$-vari\'et\'e $U$, les groupes $H^{i}(U,\calf)$ sont finis,
 et nuls pour $i $ strictement plus grand que ${\rm cd}(k) + 2 {\rm dim}(U)$
(ceci r\'esulte de la suite spectrale de Hochschild-Serre et de 
\cite{milneEC}, Th. VI.1.1. et 
Th. VI.5.5. qui valent pour $k$ alg\'ebriquement clos).
L'\'enonc\'e pour $H^i_c$ s'en d\'eduit via la suite (\ref{longlocal}) et 
le fait que les hens\'elis\'es $K_v^h$ sont \`a cohomologie galoisienne
finie, et de plus de dimension cohomologique $\leq 2$ si 
${\rm cd}(k) \leq 1$. 
\end{proof}

\begin{prop} \label{finitorsion}

Soit $k$ un corps \`a cohomologie galoisienne finie. Soient
$C$ une $k$-courbe projective, lisse, g\'eom\'etriquement int\`egre, et
$K=k(C)$.
Soit $U\subset C $ un ouvert non vide et $\T$ un tore sur $U$.

(i) Pour $i \geq 2$, les groupes 
$H^i(U,\T)$ et $H^i_c(U,\T)$ sont de torsion et de type cofini.

(ii)  Pour $i \geq 2$, les groupes 
 $H^i(U,\widehat{\T})$ et  $H^i_{c}(U,\widehat{\T})$ sont de torsion et de type cofini.

(iii)  Le groupe $H^1(U,\widehat{\T})$ est fini et le groupe 
$H^0(U,\widehat{\T})$ est de type fini.

(iv) Les groupes $H^0(U,\T)_{\tors}$ et 
$H^1(U,\T)_{\tors}$ sont de type cofini.

(v) Pour $i \geq 0$, les groupes $H^i(U,\calt)/n$ et $H^i(U,\widehat \calt)/n$
sont finis. 

\end{prop}

\begin{proof}
Pour tout sch\'ema r\'egulier $X$, le groupes $H^i(X,\G_m)$
sont de torsion pour $iÊ\geq 2$ (\cite[II, Prop. 1.4]{GrBr}.
 Par suite spectrale, on conclut qu'il en est de m\^eme
pour $H^i(X,\T)$  pour tout $X$-tore isotrivial.
Comme pour tout $n >0$ le groupe $H^i(U,{}_n \T)$ 
est fini (Lemme \ref{finitudeUfini}) et
se surjecte sur 
${}_n H^i(U,\T)$ via la suite de Kummer, on obtient aussi que 
$H^i(U,\T)$ est de type cofini pour $i \geq 2$. De m\^eme 
$H^i(U,\calt)/n$ s'injecte dans $H^{i+1}(U, _n \calt)$, donc il est 
aussi fini.

Pour $n>0$ le groupe $_n H^1(U,\T)$ est un quotient du groupe fini 
$H^1(U,_n \T)$, ce qui montre que $H^1(U,\T)_{\tors}$ est de type cofini
(mais $H^1(U,\T)$ n'est pas forc\'ement de torsion). Le groupe 
$H^0(U,\T)_{\tors}$ est \'egalement de type cofini car 
on a $_n H^0(U,\T)=H^0(U, _n T)$. 

\smallskip

Pour tout sch\'ema normal int\`egre $X$, et  $i\geq 1$,
 les groupes $H^{i}(X,\Z)$ sont de torsion
 \cite{milneADT}, Lemme II.2.10 et le groupe $H^1(X,\Z)$ est nul par
\cite{sga4}, IX.3.6. (ii).
 Comme le groupe $H^{i-1}(U,\widehat \T/n)$ 
 est fini (Lemme \ref{finitudeUfini}) et
 se surjecte sur $_n H^i(U,\widehat \T)$, 
on voit que $H^i(U,\widehat \T)$ est encore de type cofini pour tout 
$i \geq 1$. De plus, comme $H^1(U,\Z)=0$ pour tout $U$ r\'egulier, 
il existe (par 
restriction-corestriction) un $n >0$ fix\'e tel que $H^1(U,\widehat \T)=
{}_{n}H^1(U,\widehat \T)$, et ce groupe est fini puisque c'est un quotient 
de $H^0(U,\widehat \T/n)$. Le groupe $H^i(U,\widehat \calt)/n$ est \'egalement 
fini car il s'injecte dans $H^i(U, \widehat \calt /n)$, et de m\^eme 
$H^0(U,\widehat \calt)=H^0(K,\widehat T)$ est de type fini.
\smallskip 

De la suite exacte
$$\bigoplus_{v \notin U} H^{i-1}(K_v,T) \to H^i_c(U,\T) \to H^i(U,\T)$$
on d\'eduit que pour
  $i \geq 2$, le groupe $H^i_c(U,\T)$ est de torsion.
 
 De la suite exacte
 $$\bigoplus_{v \notin U} H^{i-1}(K_v,\widehat \T) \to H^i_c(U,\widehat \T) \to H^i(U,\widehat \T)$$
on d\'eduit que pour
  $i \geq 2$, le groupe 
 $H^i_c(U,,\widehat \T)$ est de torsion 
et de type cofini.
 \end{proof}

\begin{prop} \label{finitudeShai}
Soit $C$ une $\C((t))$-courbe projective, lisse, g\'eom\'e\-tri\-quement int\`egre.
 Soit $K=\C((t))(C)$ et $K_{v}$ le compl\'et\'e de $ K$ en un point de codimension 1 de $C$.

(i) Pour tout $K$-groupe ab\'elien fini $G$, et tout $i\geq 0$,  les groupes
$$\cyr{X}^{i}(K,G) := \Ker [ H^{i}(K,G) \to \prod_{v \in C^{(1)}} H^{i}(K_{v},G)]$$
sont finis. Le groupe $\Sha^1_{\omega}(K,G)$ est \'egalement fini.

(ii) Pour tout $K$-groupe de type multiplicatif $G$,  le groupe
$$\cyr{X}^1(K,G) := \Ker [ H^1(K,G) \to \prod_{v \in C^{(1)}} H^1(K_{v},G)]$$
est fini.
\end{prop}

\begin{proof} Soit $U \subset C$ un ouvert sur lequel le groupe $G$ s'\'etend en un
 groupe de type multiplicatif ${\mathcal G}$. 
D'apr\`es la proposition~\ref{extension},
  les groupes consid\'er\'es sont dans l'image
 de  $H^{i}(U,{\mathcal G}) \to H^{i}(K,G)$.
 Compte tenu du lemme \ref{finitudeUfini},
 ceci \'etablit le premier \'enonc\'e de (i) . Si maintenant l'action 
de Galois sur $G(\ov K)$ est triviale, 
la proposition~\ref{sha1sha1omega} donne que
$\Sha^1_{\omega}(K,G)=\Sha^1(K,G)$ est fini. Dans le cas g\'en\'eral, 
soit $L$ une extension finie galoisienne de groupe $\Gamma$ 
de $K$ qui d\'eploie $G$,
alors la suite exacte de restriction-inflation induit une suite 
exacte 
$$0 \to N \to  \Sha^1_{\omega}(K,G) \to \Sha^1_{\omega}(L,G),$$
o\`u $N$ est un sous-groupe du groupe fini $H^1(\Gamma,G)$, d'o\`u la finitude 
de $\Sha^1_{\omega}(K,G)$.

\smallskip
 
 Montrons (ii). Le groupe  $H^1(K,G)$
 est d'exposant fini, disons $n>0$. L'image de $H^1(U,{\mathcal G}) 
\to H^1(K,G)$
 se factorise donc par $H^1(U,{\mathcal G})/n$. 
On a une suite exacte de $U$-groupes de type multiplicatif
 $$ 1 \to \T \to {\mathcal G} \to \calf \to 1,$$
 o\`u $\T$ est un $U$-tore et $\calf$ un $U$-groupe ab\'elien fini.
 Soit $m$ un multiple de $n$ qui annule $F$. On a alors une suite exacte
   $$ H^1(U,\T)/m \to H^1(U,{\mathcal G})/n \to H^1(U,\calf)$$
  et une injection
  $$  H^1(U,\T)/m \hookrightarrow H^2(U,_m \T).$$
  Comme $H^1(U,\calf)$ et $ H^2(U,_m \T)$ sont 
finis (Lemme \ref{finitudeUfini}), ceci \'etablit (ii).
 \end{proof}
  
 \begin{rema}
 \`A la diff\'erence 
du  cas classique, celui des corps de nombres, et du cas  des corps de fonctions d'une variable sur un 
corps $p$-adique  \cite{HaSz}), sur $K=\C((t))(C)$,
 le groupe $\cyr{X}^2(K,\G_{m})$ n'est pas forc\'ement fini.
 Si la courbe $C$ admet un mod\`ele projectif et lisse sur $\C[[t]]$
 de fibre sp\'eciale $C_{0}$ de genre $g$, alors  $\cyr{X}^2(K,\G_{m}) \simeq H^1(C_{0},\Q/\Z) \simeq (\Q/\Z)^{2g}$
(voir la proposition \ref{Sha2Kmun}).
\end{rema}

\section{Dualit\'e \`a la Artin-Verdier sur un ouvert}

\begin{prop} \label{finigroup}
Pour un $U$-sch\'ema en groupes ab\'eliens finis \'etale   $\calf$, 
de dual de Cartier $\calf^0=Hom(\calf,\Q/\Z(1))$,
on a, pour $0 \leq i \leq 3$, un accouplement parfait de groupes finis
$$H^{i}(U,\calf) \times H^{3-i}_{c}(U,\calf^0) \to H^3_{c}(U,\Q/\Z(1)) = \Q/\Z(-1).$$
\end{prop}

\begin{proof}
L'argument est ici
exactement le m\^eme que dans \cite{milneEC}, Corollaire V.2.3.
Il s'agit de combiner la dualit\'e de Poincar\'e pour une vari\'et\'e projective et lisse sur un corps alg\'ebriquement clos
avec la dualit\'e en cohomologie galoisienne sur le corps $\C((t))$. Pour 
un argument d\'etaill\'e dans un cadre g\'en\'eral, voir \cite{diego}.
\end{proof}

Pour conserver la sym\'etrie, il est commode ici d'introduire 
les $1$-motifs $\M$ sur $U$ de la forme $\M=\T=[0 \to \calt]$,
dont le dual est le $1$-motif $\M^*=\widehat \calt [1]=[\widehat \calt 
\to 0]$ (par convention, dans ces complexes \`a deux termes, le terme 
de gauche est plac\'e en degr\'e $-1$ et celui de droite en degr\'e $0$).
On a alors 
une fl\`eche (dans la cat\'egorie d\'eriv\'ee born\'ee 
des faisceaux fppf sur $U$)~:
\begin{equation} \label{tensderiv}
\M \otimes^{\L} \M^* \to \G_m[1]
\end{equation}
qui induit pour $0 \leq i \leq 2$ des accouplements 
$$H^i(U,\M) \times H^{2-i}_c(U,\M^*) \to H^3_{c}(U,\G_m)=\Q/\Z(-1).$$
(Noter qu'on peut ici utiliser indiff\'eremment la cohomologie fppf 
ou \'etale, les sch\'emas en groupes qui interviennent \'etant suppos\'es 
lisses).
On dispose aussi pour tout $n >0$ de la r\'ealisation $n$-adique 
$T_{\Z/n} \M$ de $\M$, qui est un sch\'ema en groupes fini \'etale sur 
$U$, et d'un triangle exact associ\'e
\begin{equation} \label{kummertri}
T_{\Z/n} \M \to \M \to \M \to T_{\Z/n} \M[1].
\end{equation}
Le dual de Cartier de $T_{\Z/n} \M$ est $(T_{\Z/n} \M)^0=T_{\Z/n} \M^*$. 
Si $M=\calt$,
on a simplement $T_{\Z/n} \M=_n \calt$, tandis que si $M= \widehat \T [1]$,
on a $T_{\Z/n} \M=\widehat \T/n$.
La fl\`eche 
$$T_{\Z/n} \M \otimes T_{\Z/n} \M^* \to \mu_n$$ associ\'ee 
\`a cette dualit\'e est compatible (via le lemme~\ref{compatkummer}) 
avec (\ref{tensderiv}), i.e. le
diagramme 
\begin{equation} \label{compatdiag}
\begin{CD}
\M \otimes \M^* [-1] @>>> \G_m \cr
@VVV @VVV \cr 
(T_{\Z/n} \M) [1] \otimes T_{\Z/n} \M^* @>>> \mu_n[1]
\end{CD}
\end{equation}

est commutatif. 

\begin{theo}

Soit $\M$ un $1$-motif sur $U$ du type $\M=\T$ ou $\M=\widehat \T[1]$,
o\`u $\T$ est un tore. Soit $\M^*$ le dual de $\M$. Soit $l$ un nombre 
premier. Alors pour $0 \leq i \leq 2$, l'accouplement 
$$H^{i}(U,\M) \times H^{2-i}_{c}(U, \M^* ) \to \Q/\Z(-1)$$
induit des accouplements parfaits de groupes finis
$$ H^{i}(U,\M)\{l\}^{(l)} \times H^{2-i}_{c}(U,\M^* )^{(l)} \{l\} \to \Q_{l}/\Z_{l}(-1)$$ et 
$$ H^{i}_c(U,\M)\{l\}^{(l)} \times H^{2-i}(U,\M^* )^{(l)} \{l\} \to \Q_{l}/\Z_{l}(-1)$$
\end{theo}

\begin{proof}
C'est essentiellement le m\^eme argument que dans \cite{HaSz}
(Th. 1.3) ou \cite{dhsza1} (Th. 3.4). Le triangle exact (\ref{kummertri})
induit pour tout $n >0$ un diagramme commutatif (cf. (\ref{compatdiag})) 
de groupes finis (rappelons que $T_{\Z/n} \M$ est un sch\'ema en groupes
fini \'etale sur $U$), \`a lignes exactes~:
{\small
\begin{equation} \label{diagkummer}
\begin{CD}
0 @>>> H^{i-1}(U,\M)/n @>>> H^i(U,T_{\Z/n} \M) @>>> _n H^i(U,\M) @>>> 0 \cr
&& @VVV @VVV @VVV \cr
0 @>>> (_n H^{3-i}_c(U,\M^*))^D @>>> H^{3-i}_c(U,T_{\Z/n} \M^*)^D @>>> 
(H^{2-i}_c(U,\M^*)/n) ^D@>>> 0
\end{CD}
\end{equation}
}
La fl\`eche verticale du milieu est un isomorphisme par la 
Prop.~\ref{finigroup}.
En prenant $n=l^m$ (pour $m=1,2,...$) et en passant \`a la limite 
inductive sur $m$, on obtient un diagramme commutatif \`a lignes 
exactes et dont la fl\`eche verticale du milieu est un isomorphisme~:
{\small
$$
\begin{CD}
0 @>>> H^{i-1}(U,\M) \otimes \Q_l/\Z_l  @>>> \varinjlim 
H^i(U,T_{\Z/l^m} \M) @>>> H^i(U,\M)\{ l \} @>>> 0 \cr
&& @VVV @VVV @VVV \cr
0 @>>> (T_l H^{3-i}_c(U,\M^*))^D @>>> (\varprojlim 
H^{3-i}_c(U,T_{\Z/l^m} \M^*))^D @>>>
H^{2-i}_c(U,\M^*)^{(l) ^D}@>>> 0
\end{CD}
$$
}
Maintenant comme $H^{i-1}(U,\M) \otimes \Q_l/\Z_l$ est divisible, 
ce diagramme induit un isomorphisme 
$$(\varinjlim H^i(U,T_{\Z/l^m} \M))^{(l)} \simeq (H^i(U,\M)\{ l \})^{(l)}.$$
Comme le module de Tate $T_l H^{3-i}_c(U,\M^*)$ est sans torsion, on a 
de m\^eme un isomorphisme 
$$H^{2-i}_c(U,\M^*)^{(l)} \{ l \}\simeq \varprojlim
H^{3-i}_c(U,T_{\Z/l^m} \M^*)\{ l \}$$ 
qui induit par dualit\'e un isomorphisme
$$(\varprojlim H^{3-i}_c(U,T_{\Z/l^m} \M^*))\{ l \}^D \simeq 
H^{2-i}_c(U,\M^*)^{(l)} \{ l \} ^D,$$
ce qui d\'emontre finalement que la fl\`eche verticale de droite du 
diagramme pr\'ec\'edent induit un isomorphisme entre les groupes
$H^i(U,\M)\{ l \}^{(l)}$ et $H^{2-i}_c(U,\M^*)^{(l)} \{ l \} ^D$. 
On d\'emontre de m\^eme qu'on a un isomorphisme entre 
$H^i_c(U,\M)\{ l \})^{(l)}$ et $H^{2-i}(U,\M^*)^{(l)} \{ l \} ^D$
en rempla\c cant les groupes $H^i$ par les groupes $H^i_c$ et vice-versa
dans le diagramme~\ref{diagkummer}. 
\end{proof}

\medskip

En prenant successivement $\M=\T$ et $\M=\widehat{\T}[1]$, on obtient~:

\begin{cor}

Soit $i$ un entier avec $0 \leq i \leq 3$.

(i) L'accouplement 
$$H^{i}(U,\T) \times H^{3-i}_{c}(U, \widehat{\T} ) \to \Q/\Z(-1).$$
{\it induit } des accouplements parfaits de groupes finis
$$ H^{i}(U,\T)\{l\}^{(l)} \times H^{3-i}_{c}(U,\widehat{\T} )^{(l)} \{l\} \to \Q_{l}/\Z_{l}(-1)$$
et
$$ H^{i}(U,\T)^{(l)}\{l\}  \times H^{3-i}_{c}(U,\widehat{\T} ) \{l\}^{(l)} \to \Q_{l}/\Z_{l}(-1)$$

\bigskip

(ii) L'accouplement 
$$H^{i}(U,\widehat{\T}) \times H^{3-i}_{c}(U, {\T} ) \to \Q/\Z(-1).$$
{\it induit } des accouplements parfaits de groupes finis
$$ H^{i}(U,\widehat{\T})\{l\}^{(l)} \times H^{3-i}_{c}(U,\T )^{(l)} \{l\} \to \Q_{l}/\Z_{l}(-1)$$
et
$$ H^{i}(U,\widehat{\T})^{(l)}\{l\} \times H^{3-i}_{c}(U,\T )\{l\}^{(l)}  \to \Q_{l}/\Z_{l}(-1).$$

\end{cor}

\begin{rema}
{\rm Les \'enonc\'es de ce paragraphe sont valables pour un $1$-motif 
quelconque (y compris comportant une vari\'et\'e ab\'elienne); la preuve 
est identique via le fait (bien connu, mais pour lequel il est difficile 
de trouver une r\'ef\'erence) qu'on a encore le diagramme commutatif
(\ref{compatdiag}) dans ce contexte. Par contre les r\'esultats 
du paragraphe suivant ne sont pas valables pour des $1$-motifs quelconques,
faute d'avoir un bon th\'eor\`eme de dualit\'e locale pour les vari\'et\'es
ab\'eliennes sur les compl\'et\'es $K_v$ de $K$.
}
\end{rema}

\section{Groupes ab\'eliens finis : dualit\'e pour les groupes de Tate-Shafarevich}

Dans toute cette section, on pose $k=\C((t))$, et on d\'esigne 
par $C$ une $k$-courbe projective, lisse, g\'eom\'etriquement int\`egre, 
de corps des fonctions $K=k(C)$. Soit $\calf$ un groupe de type 
multiplicatif sur $U$ de dual de Cartier $\calf'=\Hom(\calf,\G_m)$.
Notons $F$ et $F'$ les fibres g\'en\'eriques respectives sur 
$\spec K$ de $\calf$ et $\calf'$.
Soit $U$ un ouvert non vide de $C$. Pour tout entier $i$ avec 
$0 \leq i \leq 3$, on a un accouplement d'Artin-Verdier 
$$H^i(U,\calf) \times H^{3-i}_c(U,\calf ') \to H^3_c(U,\G_m) \oi \Q/\Z(-1)$$
et pour toute place $v \in C^{(1)}$ un accouplement local 
$$H^i(K_v,F) \times H^{2-i}(K_v,F') \to \Br K_v \oi H^3_c(U,\G_m) 
\oi \Q/\Z(-1).$$ De m\^eme en \'echangeant $\calf$ et $\calf'$.

\begin{lem} \label{artlocal}
Avec les notations ci-dessus, on consid\`ere les fl\`eches 
$$H^i(U,\calf) \to H^{3-i}_c(U,\calf ')^D; \quad \prod_{v \in C^{(1)}} 
H^i(K_v,F) \to \bigoplus_{v \in C^{(1)}} H^{2-i}(K_v,F')^D$$
induites respectivement par l'accouplement d'Artin-Verdier et 
la somme des accouplements locaux. Alors on a un diagramme commutatif 
$$
\begin{CD}
H^i(U,\calf) @>>> \prod_{v \in C^{(1)}} H^i(K_v,F) \cr
@VVV @VVV \cr
H^{3-i}_c(U,\calf ')^D @>>> (\bigoplus_{v \in C^{(1)}} H^{2-i}(K_v,F'))^D
\end{CD}
$$
et de m\^eme en \'echangeant $\calf$ et $\calf'$.
\end{lem}

\begin{proof}
Il suffit d'appliquer la proposition~\ref{compatcoh} f) avec $j=2-i$ 
\`a chaque place $v \in C^{(1)}$, en notant aussi que la compatibilit\'e
avec les compl\'et\'es $K_v$ r\'esulte imm\'ediatement de celle avec 
les hens\'elis\'es $K_v^h$.
\end{proof}

\begin{theo}  \label{fini}
Soit $M$ un $K$-module fini; soit $M^0=\Hom_{\Z}(M,\Q/\Z(1))$.
On a un accouplement  parfait  de groupes ab\'eliens finis
$$\cyr{X}^1(K,M) \times \cyr{X}^2(K,M^0) \to \Q/\Z(-1).$$
 \end{theo}

\begin{proof}
Le reste de la section est consacr\'e \`a la d\'emonstration de ce th\'eor\`eme,
qui utilise la Prop.~\ref{compatcoh} et aussi les id\'ees de 
\cite{HaSz}, th\'eor\`eme~4.4. 
 
 Soit $M$ un $K$-module fini. Soit $U \subset C$ un ouvert non vide sur lequel $M$ (ainsi que $M^0$)
 s'\'etend en un $U$-sch\'ema en groupes ab\'eliens fini \'etale, encore not\'e $M$.
 
 Comme dans \cite[II, \S 2, p. 178]{milneEC} et  \cite[\S 3]{HaSz}
(cf. aussi la Prop.~\ref{compatcoh}, d),
on  introduit pour tout entier $i \geq 0$  les groupes 

$${\mathcal D}^{i}(U,M) = {\rm  Im}  [ H^{i}_{c}(U,M) \to H^{i}(K,M)] $$

$$D^{i}(U,M) = \Ker [H^{i}(U,M) \to \prod_{v \notin U} H^{i}(K_{v},M)]$$

Comme $H^{i}_{c}(U,M)$ et $H^{i}(U,M)$ sont finis, les groupes
${\mathcal D}^{i}(U,M) $ et $D^{i}(U,M)$ sont finis.  

\begin{lem} Soient $k=\C((t))$, $U/k$ une courbe lisse et $M$ un $U$-groupe 
fini \'etale.
Pour tout $v \in U^{(1)}$,
on a  $H^2(\calo_{v},M)=0$, et de m\^eme avec $\calo_v^h$ au lieu de $\calo_v$.
\end{lem}
\begin{proof}
On a $H^2(\calo_{v},M)=H^2(\kappa_{v},M_{\kappa_{v}})$, et ce dernier groupe est nul
car $\kappa_{v}$ est de dimension cohomologique 1. De m\^eme pour $\calo_v^h$.
\end{proof}
 
 On voit donc que dans la situation ci-dessus on a
$$D^2(U,M) =  \Ker [H^{2}(U,M) \to \prod_{v \in C^{(1)}} H^{2}(K_{v},M)].$$
Par ailleurs la proposition~\ref{compatcoh} (e) 
s'applique avec $i=1$, et fournit une 
suite exacte 
$$\bigoplus_{v \in C^{(1)}} H^0(K_v,M^0) \to H^1_c(U,M^0) \to 
{\mathcal D}^1(U,M^0) \to 0.$$

Comme les groupes $ {\mathcal D}^{1}(U,M) \subset H^{1}(K,M)$ sont finis 
et d\'ecroissent avec $U$ (Prop.~\ref{compatcoh}, d), 
  ceci implique qu'il existe un ouvert non vide $U_{0} \subset U$
 tel que pour tout ouvert  non vide $V \subset U_{0}$ de $C$, on ait
 $${\mathcal D}^{1}(V,M^0) =  {\mathcal D}^{1}(U_{0},M^0).$$
On a alors aussi ${\mathcal D}^{1}(U_0,M^0)=\Sha^1(K,M^0)$. 
En effet si $\alpha$ est un \'el\'ement de ${\mathcal D}^{1}(U_{0},M^0)$, 
alors il est pour tout $V \subset U_0$ 
dans ${\mathcal D}^{1}(V,M^0)$, donc par d\'efinition dans l'image 
de $H^1_c(V,M^0)$, ce qui implique que sa restriction $\alpha_v$ 
\`a $H^1(K_v ^h,M^0)$ (et donc aussi \`a $H^1(K_v,M^0)$) 
est nulle pour tout $v \not \in V$ via la suite exacte (\ref{longlocal}).
Ceci \'etant vrai pour tout ouvert non vide $V \subset U_0$, on obtient 
que ${\mathcal D}^{1}(U_0,M^0) \subset \Sha^1(K,M^0)$. En sens inverse, 
tout \'el\'ement de $\Sha^1(K,M^0)$ se rel\`eve dans $H^1(U_0,M^0)$ 
(Prop.~\ref{extension}, ii), en un \'el\'ement qui provient de 
$H^1_c(U_0,M^0)$ d'apr\`es la 
suite exacte (\ref{longlocal}), ce qui prouve l'inclusion en sens inverse.

Quitte \`a r\'etr\'ecir $U$, on peut supposer que $U=U_0$. 
 
\medskip

 On consid\`ere alors le diagramme  (commutatif d'apr\`es le 
lemme~\ref{artlocal}) suivant  de suites exactes
 
 $$ 
 \begin{CD}
 0  &\to &D^2(U,M) &\to&  H^2(U,M) &\to& \prod_{v \in C^{(1)}}    H^2(K_{v},M) \cr
 &&  @VVV @VVV @VVV \cr
 0 &\to& {\mathcal D}^{1}(U,M^0)^D &\to & H^{1}_{c}(U,M^0)^D &\to &(\bigoplus_{v \in C^{(1)} }  H^0(K_{v},M^0))^D .
 \end{CD}
 $$

D'apr\`es la proposition \ref{finigroup}, la fl\`eche m\'ediane est un isomorphisme.
Les th\'eor\`emes de dualit\'e locale (Prop. \ref{dualitelocalefinie}) donnent que
la  fl\`eche de droite est un isomorphisme.   
  
  La fl\`eche induite $$ D^2(U,M)  \to {\mathcal D}^{1}(U,M^0)^D$$ est donc un isomorphisme de groupes
 ab\'eliens finis.
 On a  ${\mathcal D}^{1}(U,M^0) = \cyr{X}^1(K,M^0)$.
Il reste \`a voir le lien entre $ D^2(U,M) $ et $\cyr{X}^2(K,M)$.
En utilisant la proposition \ref{extension} on voit que pour $V \subset U$ avec $M$ un $U$-sch\'ema en groupes finis
\'etales,  la fl\`eche de restriction
$H^2(U,M) \to H^2(V,M)$ induit une surjection de groupes finis
$D^2(U,M) \to D^2(V,M)$.
En consid\'erant les images dans $H^2(K,M)$, on voit qu'il existe un ouvert 
fixe $U_{1} \subset C$ tel que pour tout ouvert non vide $U$ de $C$ les fl\`eches
$$D^2(U_{1},M) \to D^2(U\cap U_{1},M) \to \cyr{X}^2(K,M)$$
soient des isomorphismes. Quitte \`a restreindre encore $U$, on peut donc 
supposer que $D^2(U,M)=\Sha^2(K,M)$, ce qui termine la preuve.
\end{proof}
 
\section{Tores : dualit\'es pour les groupes de Tate-Shafarevich}

Soit $U\subset C $ un ouvert non vide et $\T$ un tore sur $U$ de fibre 
g\'en\'erique $T$. Pour tout groupe ab\'elien $A$, on note 
$\ov A$ le quotient de $A$ par son sous-groupe divisible maximal.

\begin{theo}\label{dualitesha2tore}
On a un accouplement parfait de groupes finis
$$\cyr{X}^1(K,\widehat{T})  \times \overline{\cyr{X}^2({K,T)}} \to \Q/\Z(-1).$$
\end{theo}

\begin{theo} \label{dualitesha1tore}
On a un accouplement parfait de groupes finis
$$\cyr{X}^1(K,T)  \times \overline{\cyr{X}^2({K,\widehat T)}} \to \Q/\Z(-1).$$
\end{theo}

\begin{proof} On suit la m\'ethode de \cite{HaSz}, th\'eor\`eme~4.1, en 
commen\c cant par des observations qui serviront pour les deux 
\'enonc\'es.
Pour $i=1,2$, et $\T$ tore sur $U$, on va encore utiliser les groupes
$$\sD^{i}(U,\T)  = {\rm Im} \, [H^{i}_{c}(U,\T)  \to H^{i}(K,T)].$$

Pour $i=1$, c'est un  groupe fini car $H^1(K,T)$ est d'exposant fini
et l'image se factorise donc par un quotient $H^{1}_{c}(U,\T)/n \subset H^2_{c}(U,{}_{n}\T)$.
Les groupes   $H^{2}_{c}(U,\T)$ et $H^{2}(U,T)$  sont de torsion, de  type cofini.
Ceci implique que les groupes $\sD^{2}(U,\T)\{l\}$ sont de  type cofini.
Fixons un nombre premier $l$.
Comme dans la prop. 3.6 de  \cite{HaSz} il existe un ouvert $U_{0}$
(d\'ependant de $l$) tel que pour tout ouvert  non vide $V \subset U_{0}$
on ait $\sD^{i}(V,\T)\{l\}= \sD^{i}(V,\T)\{l\} 
\subset H^{i}(K,T)$ pour $i=1,2$. Quitte \`a restreindre $U$, on peut 
supposer $U_0=U$. 
Comme dans la Prop. 3.6 de   \cite{HaSz}, on a 
alors 
$$\sD^{i}(U,T)\{l\} = \cyr{X}^{i}(K,T)\{l\}$$
(et de m\^eme pour tout ouvert non vide inclus dans $U$), 
puisque si $V \subset U$ tout \'el\'ement de $H^i_c(V,\calf)$ a une image nulle 
dans $H^i(K_v,F)$ pour $v \not \in V$.

\smallskip

On va aussi utiliser
$$\sD^{i}(U,\widehat{\T})  = {\rm Im} \, [H^{i}_{c}(U,\widehat{\T})  \to H^{i}(K,\widehat{T})].$$
Pour $i=1$ c'est  un groupe fini.
 Pour $i \geq 2$ c'est un groupe de torsion de  type cofini.
 L\`a encore on peut supposer que pour tout ouvert non vide $V \subset U$
on a $\sD^{i}(V,\widehat{\T})\{l\}= \sD^{i}(U,\widehat{\T})\{l\} \subset H^{i}(K,\widehat{T})$ pour $i=1,2$.

Partant de l\`a on obtient comme ci-dessus 
$$\sD^{i}(U, \widehat{T})) \{l\} = \cyr{X}^{i}(K,\widehat{T})\{l\},$$
et de m\^eme pour tout ouvert non vide inclus dans $U$.
Pour $i=2$, ce groupe n'est pas forc\'ement fini.

\bigskip

{\it Preuve du th\'eor\`eme~\ref{dualitesha2tore}.}

La prop.~\ref{compatcoh} e) s'applique \`a $\T$ pour $i=2$ (noter qu'ici 
$H^2(\calo_v,\T)=0$ car le corps r\'esiduel $\kappa(v)$ de $\calo_v$ est de 
dimension cohomologique $1$), et donne la  suite exacte
de groupes de torsion
$$\bigoplus_{v \in C^{(1)}} H^1(K_{v},T) \to H^2_{c}(U,\T) \to \sD^2(U,\T) \to 0,$$

On a des suites exactes induites
$$\bigoplus_{v \in C^{(1)}} H^1(K_{v},T)\{l\} \to H^2_{c}(U,\T)\{l\} \to \sD^2(U,\T)\{l\} \to 0.$$

 On a des suites exactes
$$ 0 \to  H^2_{c}(U,\T)\{l\}_{\div}  \to H^2_{c}(U,\T)\{l\} \to \overline{H^2_{c}(U,\T)\{l\}} \to 0$$
et 
$$0 \to \sD^2(U,\T)\{l\}_{\div} \to \sD^2(U,\T)\{l\} \to \overline{\sD^2(U,\T)\{l\}} \to 0.$$
En prenant les duaux $A \mapsto A^D$ (\`a coefficients $\Q_{l}/\Z_{l}(-1)$),
on obtient deux suites exactes 
$$ 0 \to \overline{H^2_{c}(U,\T)\{l\}} ^D \to  H^2_{c}(U,\T)\{l\}^D \to [H^2_{c}(U,\T)\{l\}_{\div}]^D \to 0$$
et
$$0  \to\overline{\sD^2(U,\T)\{l\}}^D \to \sD^2(U,\T)\{l\}^D \to [\sD^2(U,\T)\{l\}_{\div}]^D \to  0.$$
Dans ces deux derni\`eres  suites, les groupes de gauche sont finis, les groupes de droite sont des $\Z_{l}$-modules
de type fini sans torsion.

\bigskip

On d\'efinit
 $$D^{1}_{sh}(U,\widehat{\T}) = \Ker [H^{1}(U,\widehat{\T}) \to \prod_{v \in C^{(1)}} H^{1}(K_{v},\widehat{T})].$$
 
 Le groupe $H^{1}(U,\widehat{\T})$ est fini, donc aussi $D^{1}_{sh}(U,
\widehat{\T}) $.

\bigskip

On a le diagramme (commutatif par le lemme~\ref{artlocal})
 
 {
$$\begin{array}{ccccccccccc}
0  & \to &  D^{1}_{sh}(U,\widehat{\T}) \{l\} & \to & H^{1}(U,\widehat{\T})\{l\} & \to & \prod_{v \in C^{(1)}} H^{1}(K_{v},\widehat{T})\{l\} \cr
&& \downarrow && \downarrow& &  \downarrow   \\
0 & \to & \sD^2(U,\T)\{l\}^D & \to & H^2_{c}(U,\T)\{l\}^D &\to & [\bigoplus_{v \in C^{(1)}} H^1(K_{v},T)\{l\}]^D
\end{array}
$$}

Dans ce diagramme, les groupes $D^{1}_{sh}(U,\widehat{\T})$ et 
$H^{1}(U,\widehat{\T})$ sont finis.

Il r\'esulte alors de ce qui pr\'ec\`ede, et du fait que les groupes de droite sont d'exposant fini,
que ce diagramme se factorise en un diagramme commutatif de suites exactes : 
 {
$$\begin{array}{ccccccccccc}
0  & \to &  D^{1}_{sh}(U,\widehat{\T}) \{l\}  & \to & H^1(U,\widehat {\T}) \{l\}& \to & \prod_{v \in C^{(1)}} H^{1}(K_{v},\widehat{T})\{l\} \cr
&& \downarrow && \downarrow& &  \downarrow   \\
0 & \to & \overline{\sD^2(U,\T)\{l\}}^D & \to & \overline{H^2_{c}(U,\T)\{l\}}^D &\to & [\bigoplus_{v \in C^{(1)}} H^1(K_{v},T)\{l\}]^D.
\end{array}
$$}

On a vu que l'on a des accouplements parfaits de groupes finis
$$ H^{i}(U,\widehat{\T})^{(l)}\{l\} \times H^{3-i}_{c}(U,\T )\{l\}^{(l)}  \to \Q_{l}/\Z_{l}(-1).$$
Appliquant ceci \`a $i=1$ et utilisant le fait que $ H^{1}(U,\widehat{T)}$ est fini, on voit que la fl\`eche verticale m\'ediane
est un isomorphisme. Par les dualit\'es locales, la fl\`eche verticale de droite est un isomorphisme.
On obtient donc une dualit\'e parfaite de groupes finis :
$$ D^{1}_{sh}(U,\widehat{\T}) \{l\} \times \overline{\sD^2(U,\T)\{l\}} \to \Q_{l}/\Z_{l}(-1).$$
Or on a d\'ej\`a vu qu'on avait 
$$\sD^{2}(U,\T)\{l\} = \cyr{X}^{2}(K,T)\{l\}.$$
 Pour $\cyr{X}^1(K,\widehat{T})$, on utilise le lemme suivant. 
 
\medskip

\begin{lem}
 Pour tout ouvert $V \subset U$, la fl\`eche de restriction
$H^1(U,\widehat{\T} ) \to H^1(V,\widehat{\T})$ est un isomorphisme, et il en est de m\^eme
de la restriction $H^1(U,\widehat{\T} ) \to H^1(K,\widehat{T})$.
On a donc
$\cyr{X}^1(K,\widehat{T})=D^{1}_{sh}(U,\widehat{\T})$.
\end{lem}

\begin{proof}  Soit $U' \to U$ un rev\^etement fini \'etale galoisien int\`egre, de groupe $G$, d\'eployant $\widehat{\T}$.
On a alors la suite exacte
$$ 0 \to H^1(G,\widehat{\T}(U')) \to H^1(U,\widehat{\T}) \to H^1(U', \widehat{\T})$$
qui, comme $H^1(U',\Z)=0$ (car $U'$ est lisse et en particulier normal) donne un isomorphisme
$H^1(G,\widehat{\T}(U')) \simeq H^1(U,\widehat{\T})$. 
\end{proof}

On conclut finalement qu'on a pour chaque nombre premier $l$ une
dualit\'e parfaite entre les groupes finis 
$\overline{\cyr{X}^{2}(K,T)\{l\}}$ et $\cyr{X}^1(K,\widehat{T})\{l\}$, 
ce qui termine la preuve du th\'eor\`eme~\ref{dualitesha2tore} en raisonnant 
composante $l$-primaire par composante $l$-primaire.

\bigskip

{\it Preuve du th\'eor\`eme~\ref{dualitesha1tore}.} 

 On peut appliquer la Prop.~\ref{compatcoh} e) avec $i=1$ \`a 
$\widehat \calt$, car pour $v \in U$ on a un isomorphisme
$H^1(\calo_v,\widehat \calt) \to H^1(K_v,\widehat T)$ via la suite
de restriction-inflation et le fait que le groupe d'inertie de $K_v$ 
agit trivialement sur $\widehat T$. On en d\'eduit une
suite exacte de groupes de torsion
$$\bigoplus_{v \in C^{(1)}} H^1(K_{v},\widehat{T}) \to H^2_{c}(U,\widehat{\T}) \to \sD^2(U,\widehat{\T}) \to 0.$$

On d\'efinit
 $$D^{1}_{sh}(U,\T) = \Ker [H^{1}(U,\T) \to \prod_{v \in C^{(1)}} H^{1}(K_{v},T)].$$
 (rappel : ceci n'est en g\'en\'eral pas un groupe de torsion.)

On a le diagramme commutatif
 
 {
$$\begin{array}{ccccccccccc}

0  & \to &  D^{1}_{sh}(U,\T) \{l\} & \to & H^{1}(U,\T)  \{l\} & \to & \prod_{v \in C^{(1)}} H^{1}(K_{v},T)\{l\} \cr
&& \downarrow && \downarrow& &  \downarrow   \\
0 & \to & \sD^2(U,\widehat{\T})\{l\}^D & \to & H^2_{c}(U,\widehat{\T})\{l\}^D &\to & [\bigoplus_{v \in C^{(1)}} H^1(K_{v},\widehat{\T})\{l\}]^D

\end{array}
$$}

Le th\'eor\`eme donne un accouplement parfait de groupes finis
$$ H^{1}(U,\T)\{l\}^{(l)} \times H^{2}_{c}(U,\widehat{\T} )^{(l)} \{l\} \to \Q_{l}/\Z_{l}(-1).$$
Compte tenu du fait que $H^{2}_{c}(U,\widehat{\T} )$ est un groupe de torsion de  type cofini,
ceci se lit aussi comme
$$ \overline{H^{1}(U,\T)\{l\}  } \times \overline{H^{2}_{c}(U,\widehat{\T} ) \{l\}} \to \Q_{l}/\Z_{l}(-1).$$

Comme tout accouplement entre groupes de torsion passe au quotient par les sous-groupes divisibles maximaux,
et comme les groupes de droite dans le diagramme sont annul\'es par un entier fixe, ce diagramme induit
un diagramme commutatif de suites exactes
 {
$$\begin{array}{ccccccccccc}

0  & \to &  \overline{D^{1}_{sh}(U,\T) \{l\} } & \to & \overline{H^{1}(U,\T)  \{l\} } & \to & \prod_{v \in C^{(1)}} H^{1}(K_{v},T)\{l\} \cr
&& \downarrow && \downarrow& &  \downarrow   \\
0 & \to & \overline{\sD^2(U,\widehat{\T})\{l\}}^D & \to & \overline{H^2_{c}(U,\widehat{\T})\{l\}}^D &\to & [\bigoplus_{v \in C^{(1)}} H^1(K_{v},\widehat{\T})\{l\}]^D 

\end{array}
$$}
  qui induit donc un isomorphisme de groupes finis
$$ \overline{D^{1}_{sh}(U,\T) \{l\} } \simeq  \overline{\sD^2(U,\widehat{\T})\{l\}}^D $$
soit encore une dualit\'e parfaite de groupes finis
$$ \overline{D^{1}_{sh}(U,\T) \{l\} } \times  \overline{\sD^2(U,\widehat{\T})\{l\}} \to \Q_{l}/\Z_{l}(-1).$$

On a vu plus haut qu'on pouvait choisir $U$ de sorte que  
$$\sD^2(U,\widehat{\T})\{l\} \simeq \cyr{X}^2(K, \widehat{T})\{l\},$$
et de m\^eme pour tout ouvert non vide $V \subset U$.

Le groupe fini  $\cyr{X}^{1}(K,T)\{l\}$ est la limite inductive des 
$D^{1}_{sh}(V,\T) \{l\} $ pour $V \subset U$.
Comme l'image d'un groupe divisible est un groupe divisible, 
$\cyr{X}^{1}(K,T)\{l\}$ est aussi la limite inductive des groupes $\overline{D^{1}_{sh}(V,\T) \{l\} }$.

Ceci d\'emontre :
 
Pour tout $l$ premier, on a une dualit\'e parfaite de groupes finis entre
$\cyr{X}^{1}(K,T)\{l\}$ et $\overline{\cyr{X}^2(K, \widehat{T})\{l\}}$ 
par son sous-groupe divisible maximal.

Le groupe $\cyr{X}^{1}(K,T)\subset H^1(K,T)$  est annul\'e par multiplication par le degr\'e
de toute extension de corps de $K$ d\'eployant le $K$-tore $T$, d'o\`u on 
d\'eduit le r\'esultat.
\end{proof}

\begin{rema}
L'exemple~\ref{contreexrationnel} montre que $\Sha^1(K,T)$ peut \^etre non 
nul (un exemple plus explicite est d\'evelopp\'e dans la section \ref{exempledetaille}).
Montrons par ailleurs 
qu'il existe un r\'eseau $\widehat{T} $ avec $\Sha^1(K, \widehat{T} )
\neq 0$.
  Soit $E/\C$ une courbe elliptique. Soit $K=\C((t))(E)$.
D'apr\`es le corollaire~\ref{corshaZ/n}, on a
$\Sha^1(K,\Q/\Z) \simeq (\Q/\Z)^2$.
On a donc $\Sha^2(K,\Z)= (\Q/\Z)^2$.
On part d'une  classe $\xi$  non nulle dans $\Sha^2(K,\Z)$.
Il existe une extension $L$ finie galoisienne, de groupe $G$, telle que
$\xi$ s'annule dans $H^2(L, \Z)$.

La fl\`eche $\Z \to \Z[G]$ d\'efinie par la norme $N_G=\sum_{g \in G}g$
induit une suite exacte de $G$-modules
$$ 0 \to \Z \to \Z[G]  \to \widehat{T} \to 0,$$
o\`u $\widehat{T}$ est un r\'eseau.

On d\'eduit de cette suite un plongement
$$\Sha^1(K,  \widehat{T} ) \hookrightarrow \Sha^2(K, \Z),$$
et $\xi \in \Sha^2(K, \Z)$ est dans l'image de $\Sha^1(K,  \widehat{T} )$,
qui est donc  non nul.

\end{rema}

\section{Lien avec l'obstruction de r\'eciprocit\'e} \label{neuf}

On discute ici l'analogue du paragraphe 5 de \cite{HaSz}, et de consid\'erations classiques
sur les corps de nombres (voir \cite[Chap. 6]{Sk}).

\medskip

Soient  $K$ un corps, $G=\Gal({\overline K}/K)$ et $Y$ un espace principal homog\`ene d'un $K$-tore $T$.
Comme on a $\Pic({\overline Y})=0$,
de la suite spectrale
$$E_{2}^{pq} = H^p(K,H^q({\overline Y}, \G_{m})) \Longrightarrow  H^n(Y,\G_{m} )$$
on tire un isomorphisme
$$ H^2(G,{\overline K}[Y]^{\times})) \oi \Br_{1}(Y).$$
o\`u $\Br_{1}(Y) = \Ker [\Br(Y) \to \Br({\overline Y})]$.
On a la suite exacte de $G$-modules
$$ 0 \to {\overline K}^{\times} \to {\overline K}[Y]^{\times} \to \widehat{T} \to 0.$$
On a donc des suites exactes
$$ \Br(K) \to \Br_{1}(Y) \to  H^2(G,{\widehat T}) \to H^3(G,{\overline K}^{\times}).$$
Si l'on a ${\rm cd}(K) \leq 2$, alors on a $H^3(G,{\overline K}^{\times})=0$.

\bigskip

Soit $k=\C((t))$.
Supposons d\'esormais 
que $K=k(C)$ est le corps des fonctions d'une $k$-courbe $C$
projective, lisse, g\'eom\'etriquement connexe. On note $\Omega$ l'ensemble des
points ferm\'es de $C$.
Le corps $K$ et ses compl\'et\'es $K_{v}$ aux points ferm\'es $v \in \Omega$
 sont de dimension cohomologique 2.
\medskip

On renvoie \`a la Proposition \ref{cohocourbes} pour les \'enonc\'es  suivants. 
Pour chaque place $v$, on a $\Br(K_{v}) \oi H^1(\kappa_{v},\Q/\Z) \simeq \Q/\Z(-1)$.
On a une suite exacte
$$ 0 \to \Br(C) \to \Br(K) \to \bigoplus_{v \in \Omega} \Br(K_{v}) \to \Q/\Z(-1)  \to 0.$$
La nullit\'e de l'application compos\'ee $$ \Br(K) \to \bigoplus_{v \in \Omega} \Br(K_{v}) \to \Q/\Z(-1)$$
est un cas particulier de loi de r\'eciprocit\'e.

\subsection{Sous-groupes du groupe de Brauer}\label{sgpbrauer}

Soient $Y$ une $K$-vari\'et\'e lisse g\'eom\'etriquement int\`egre et $Y_{c}$ une $K$-compactification lisse.
Pour $v \in \Omega$, on note $Y_{v}=Y\times_{K}K_{v}$.

Le groupe $\Br_{1}(Y)=\ker [\Br Y \to \Br  {\overline{Y}}]$ contient les sous-groupes suivants : 

(a)  Le  groupe $\Br(Y_{c})=\Br_{1}(Y_{c}) \subset  \Br_{1}(Y)$.

(b) Le groupe $B(Y) \subset \Br_{1}(Y) $
 form\'e des \'el\'ements de $\Br_{1}(Y)$  dont l'image dans  
$\Br_{1}(Y_{v})$ appartient \`a
  l'image de $\Br(K_{v})$ pour toute
place $v \in C^{(1)}$.

(c)  Le groupe $B_{\omega}(Y) \subset \Br_{1}(Y) $
 form\'e des \'el\'ements de $\Br_{1}(Y)$  dont l'image dans $\Br_{1}(Y_{v})$
appartient \`a l'image de $\Br(K_{v})$ pour presque toute place 
$v \in C^{(1)}$. Un argument de bonne r\'eduction,
et la nullit\'e des groupes $\Br(O_{v})$ pour $O_{v}$ 
compl\'et\'e de l'anneau local
en $v \in C$, de corps r\'esiduel une extension finie de $\C((t))$, montre que $B_{\omega}(Y)$
est form\'e des   \'el\'ements de $\Br_{1}(Y)$ dont l'image est 
nulle dans $\Br_{1}(Y_{v})$ pour presque toute place $v \in C^{(1)}$.

\medskip

On note $\Br_a Y$ le quotient de $\Br_1 Y$ par l'image de $\Br \, K$.
On note  $\Br_a Y_c$ le quotient de $\Br_1 Y_c  $ par 
l'image de $\Br \, K$.

\medskip
Consid\'erons le cas o\`u $Y$ est un espace principal homog\`ene sous un $K$-tore $T$.
Pour presque toute place $v$, le groupe de d\'ecomposition de l'extension minimale
d\'eployant le $K$-tore $T$ est cyclique (voir le paragraphe \ref{lescorps}). 
De plus pour presque toute
place $v$ on a $Y_{v} \simeq T_{v}$. Le groupe de Brauer d'une compactification lisse d'un tore d\'eploy\'e par 
une extension cyclique du corps de base est r\'eduit au groupe de Brauer du corps de base (\cite{sansuc}, Th. 9.2.ii), r\'esultat d\^u aussi \`a 
Voskrensenski\v{\i}, cf. \cite{voskbook} Chap. 4, Section 11.6,
Cor. 3 page 122).
Le groupe $\Br(Y_{c})=\Br_1(Y_{c})$ est donc un sous-groupe du groupe  $B_{\omega}(Y)$.

On a   le diagramme   :
$$\begin{array}{ccccccccccc}
   B(Y) & \subset & B_{\omega}(Y) & \subset &  \Br_{1}(Y) \cr
 \uparrow && \uparrow{\subset} && \cr
 \Br(K)& \to & \Br(Y_{c}). &&   
  \end{array}
  $$

Par passage au quotient par l'image de $\Br(K)$ dans $B(Y)$, resp. dans  $B_{\omega}(Y)$,
on d\'efinit le groupe $ \cyr{B}(Y)$,  resp. le groupe $\cyr{B}_{\omega}(Y)$, et on obtient le
 diagramme :
$$\begin{array}{ccccccccccc}
  \cyr{B}(Y) & \subset & \cyr{B}_{\omega}(Y) & \subset &  \Br_{a}(Y) \cr
  && \uparrow{\subset} && \cr
 \ &   & \Br_a(Y_{c}) &&   
  \end{array}
  $$

D'apr\`es les pr\'eliminaires de ce paragraphe, on a  un diagramme commutatif de suites exactes
$$\begin{array}{ccccccccccc}
&   &  \Br(K)& \to & \Br_{1}(Y) & \to &  H^2(K,{\widehat T}) & \to & 0 \cr
&& \downarrow && \downarrow && \downarrow &&   \cr
 &&  \bigoplus_{v\in \Omega} \Br(K_{v}) & \to & \bigoplus_{v\in \Omega} \Br_{1}(Y_{v}) & \to & \bigoplus_{v\in \Omega}H^2(K_{v},{\widehat T}) & \to & 0
  \end{array}
  $$

On a donc un isomorphisme
$\cyr{B}(Y)  \oi \cyr{X}^2(K,\widehat{T})$ et un isomorphisme $\cyr{B}_{\omega}(Y)  \oi \cyr{X}^2_{\omega}(K,\widehat{T})$.
Pour $i\in \N$,  et un module galoisien $M$
 sur le corps $K$, on note  $\cyr{X}^{i}_{\rm cyc}(K,M)$ le noyau 
de la restriction
 $$H^{i}(G,M) \to \prod_{g \in G} H^{i}(<g>,M)$$
aux sous-groupes ferm\'es procycliques.
D'apr\`es   \cite[Prop. 9.5]{CTSflasque}, pour $T_{c}$ une compactification lisse de $T$,
on a $\Br(T_{c})/\Br(K) \simeq \cyr{X}^{2}_{\rm cyc}(K,\widehat{T})$. 
Comme ${\rm cd}(K) \leq 2$, 
d'apr\`es  \cite[Lemme 6.8]{sansuc} et  \cite[Prop. 2.2]{CTBordeaux},
 il existe un isomorphisme naturel  $ \Br_{1}(Y) \oi   \Br_{1}(T)$
compatible aux isomorphismes
 $ \Br_{a}(Y)\oi H^2(K,\widehat{T})$  et $ \Br_{a}(T) \oi H^2(K,\widehat{T})$,
 et induisant un isomorphisme $\Br_a(Y_{c}) \oi  \Br_a(T_{c})$.

 On obtient ainsi un isomorphisme $ \Br_{a}(Y_{c}) \oi  \cyr{X}^2_{\rm cyc}(K,\widehat{T})$.
Le diagramme ci-dessus se r\'ecrit alors :
$$\begin{array}{ccccccccccc}
  \cyr{X}^2(K,\widehat{T})& \subset & \cyr{X}^2_{\omega}(K,\widehat{T}) & \subset &  H^2(K,,\widehat{T}) \cr
  && \uparrow{\subset} && \cr
 \ &   & \cyr{X}^2_{\rm cyc}(K,\widehat{T}) &&   
  \end{array}
  $$

\begin{remas}
 
\`A la diff\'erence du cas o\`u $K$ est un corps global :

(i)  M\^eme si $Y$ a des points dans tous les
$K_{v}$ il n'est pas clair que $\Br(K)$ s'injecte dans $\Br(Y)$;

(ii) Le groupe $\cyr{X}^2_{\rm cyc}(K,\widehat{T})$ est fini (propri\'et\'e valable sur tout corps), mais
les groupes $\cyr{B}(Y) =\cyr{X}^2 (K,\widehat{T})$ et $\cyr{B}_{\omega}(Y) =\cyr{X}^2_{\omega} (K,\widehat{T}) $
peuvent \^etre infinis. Pour $Y=T=\G_{m}$, on a $$\cyr{X}^2 (K,\Z)= \cyr{X}^1(K,\Q/\Z)=  \cyr{X}^1_{\omega}(K,\Q/\Z)=H^1(C_{0},\Q/\Z),$$
o\`u $C_{0}$ est la courbe sur $\C$ fibre sp\'eciale d'un mod\`ele minimal de $C$ au-dessus de $\C[[t]]$ (voir le corollaire
\ref{shadivisible} et la d\'emonstration du corollaire \ref{sha1sha1omega}).

(iiii) Un \'el\'ement de  $ {B}_{\omega}(Y) $, ou m\^eme de  ${B}(Y) $, n'est pas n\'ecessairement dans $\Br(Y_{c})$.
\end{remas}

\subsection{Accouplement avec les points ad\'eliques}\label{accouplements}

Soit $Y$ une $k$-vari\'et\'e lisse g\'eom\'etriquement int\`egre,
poss\'edant des $K_{v}$-points pour tout $v \in \Omega$.
Soit $Y_{c}$ une 
$K$-compactifcation lisse de $Y$.

On a les inclusions

$$Y(\A_{K}) \subset \prod_{v} Y(K_{v}) \subset \prod_{v}Y_{c}(K_{v})=Y_{c}(\A_{K}).$$

En utilisant les isomorphismes $\partial_{v} : \Br(K_{v}) \oi  \Q/\Z(-1)$,
on peut d\'efinir plusieurs accouplements
\`a la Brauer--Manin
$$Y(\A_{K}) \times \Br(Y) \to \Q/\Z(-1)$$
$$ \prod_{v} Y(K_{v}) \times B_{\omega}(Y) \to \Q/\Z(-1)$$
$$ \prod_{v} Y_{c}(K_{v}) \times B(Y) \to \Q/\Z(-1)$$
envoyant le couple form\'e d'une famille $\{P_{v}\}_{v \in \Omega}$ et d'un \'el\'ement  $\alpha$
sur la  somme (qui n'a qu'un nombre fini de termes non nul)
$$ \sum_{v} \partial_{v}(\alpha(P_{v})) \in \Q/\Z(-1).$$
Comme on a le complexe 
$$\Br(K) \to \bigoplus_{v \in \Omega} \Br(K_{v}) \to \Q/\Z(-1),$$
ces accouplements ne d\'ependent que des quotients $\Br(Y)/\Br(K)$,
resp.   $\cyr{B}_{\omega}(Y)$, resp.  $\cyr{B}(Y)$.

Le dernier accouplement ne d\'epend pas du terme de gauche. 
En effet un \'el\'ement de $B(Y)$ provient d'un \'el\'ement    $\alpha \in \Br(Y)$ 
tel que pour tout $v \in \Omega$ il existe $\beta_{v} \in \Br(K_{v})$
dont l'image dans $\Br(Y_{v})$ co\"{\i}ncide avec l'image de $\alpha$,
la famille $\{\beta_{v}\}$ \'etant d\'efinie \`a addition de l'image diagonale
d'un \'el\'ement de $\Br(K)$. La somme
 $$ \sum_{v} \partial_{v}(\alpha(P_{v}))=   \sum_{v} \partial_{v}(\beta_{v})$$
est alors clairement ind\'ependante de la famille  $\{P_{v}\}$.
Ainsi  $Y$ d\'efinit un homomorphisme
$$  \rho_{Y}: \cyr{B}(Y) \to \Q/\Z(-1).$$
 
 On note $Y(\A_{K})^{\Br(Y)}$ le noyau \`a gauche de la premi\`ere fl\`eche,
 et on emploie des notations semblables pour les autres accouplements.
 Du  complexe ci-dessus 
on tire les inclusions
{\small
$$Y(K) \subset  Y(\A_{K})^{\Br(Y)}  \subset  [\prod_{v\in \Omega} Y(K_{v})]^{ \cyr{B}_{\omega}(Y) } \subset 
[\prod_{v\in \Omega} Y(K_{v})]^{  \cyr{B} (Y) }  \subset Y_{c}(\A_{K})^{\cyr{B}(Y)}.$$
}

\begin{prop}\label{compatiblelementaire}
Soit $T$ un $K$-tore et $Y$ un $K$-espace principal homog\`ene
poss\'edant des points dans tous les $K_{v}$. Soit 
   $\xi \in \cyr{X}^{1}(K,T)$ la classe de $Y$. Soit $\beta \in \cyr{B}(Y)$
   et $\gamma \in \cyr{X}^{2}(K,\widehat{T})$ son image par  l'isomorphisme
   $\cyr{B}(Y) \oi \cyr{X}^{2}(K,\widehat{T})$. On a l'\'egalit\'e :
   $$ \rho_{Y}(\beta) = <\xi, \gamma > \in \Q/\Z(-1),$$
   o\`u $<\xi,\gamma>$ est la valeur sur $(\xi,\gamma)$ de l'accouplement
   $$ \cyr{X}^{1}(K,T) \times \cyr{X}^{2}(K,\widehat{T}) \to \Q/\Z(-1)$$
  qui fait l'objet du  th\'eor\`eme \ref{dualitesha1tore}.
    \end{prop}

\begin{proof}
La preuve est essentiellement identique \`a celle de 
\cite{dhsza2}, sections 3 et 4 (qui vaut sur un corps 
de nombres, pour un $1$-motif quelconque \`a la place du tore $T$); nous nous
bornons ici \`a en rappeler les grandes lignes. On commence par
\'etendre $T$ en un $U$-tore $\calt$, et on 
rel\`eve la classe $\xi \in \Sha^1(T)$ en une classe $\xi_U \in H^1(U,\calt)$
pour $U$ assez petit. De m\^eme on peut relever $\gamma$ en $\gamma_U \in 
H^2_c(U,{\widehat \calt})$. La proposition~3.3. de \cite{dhsza2} s'applique 
encore dans notre cadre, avec une preuve analogue (si ce n'est qu'il faut 
remplacer, dans le diagramme (7) de \cite{dhsza2}, le groupe $\Br k$ par 
son image dans $\Br_1 Y$) et donne 
$$\rho_Y(\beta)={\mathcal E}_Y \cup \gamma_U$$
o\`u ${\mathcal E}_Y \in {\rm Ext}^1_U({\widehat \calt}, \G_m)$ 
est la classe d'une extension dont la fibre g\'en\'erique sur $\spec K$ 
est celle de 
\begin{equation} \label{extph}
0 \to \ov K^* \to \ov K[Y]^* \to \ov K[Y]^*/\ov K^*=\widehat T \to 0.
\end{equation}
Ensuite, les lemmes~2.3.7. et 2.4.3. de de \cite{Sk} donnent
que la classe de l'extension 
(\ref{extph}) dans ${\rm Ext}^1_K(\widehat T,\G_m)=H^1(K,T)$ est pr\'ecis\'ement
celle de $\xi$. Quitte \`a r\'etr\'ecir $U$, on peut donc supposer 
que ${\mathcal E}_Y=\xi_U$, ce qui donne le r\'esultat d'apr\`es la 
d\'efinition de l'accouplement du th\'eor\`eme~\ref{dualitesha1tore}.

\end{proof}

On en d\'eduit l'analogue du th\'eor\`eme de Voskrenski\v{\i}-Sansuc 
sur les corps de nombres et du th\'eor\`eme~5.1 de \cite{HaSz} 
sur le corps de fonctions d'une courbe au-dessus d'un corps $p$-adique~:

\begin{cor}\label{elementairelaseule}  
Soit $Y$ un espace principal homog\`ene sous un $K$-tore.
Supposons que $Y$ poss\`ede des points dans tous les $K_{v}$.
L'obstruction \`a l'existence d'un point rationnel sur $Y$
associ\'ee au groupe  $B(Y)$  et \`a la loi de r\'eciprocit\'e sur $C$
est la seule obstruction \`a l'existence d'un point rationnel sur $Y$.
\end{cor}
\begin{proof}
 L'hypoth\`ese assure que l'application  $\rho_{Y}$ est nulle. La proposition 
  \ref{compatiblelementaire}  donne alors que la classe de $Y$ dans $ \cyr{X}^{1}(K,T)$
  est orthogonale \`a $\cyr{X}^{2}(K,\widehat{T})$ pour un accouplement qui,
  d'apr\`es le le th\'eor\`eme \ref{dualitesha1tore}, est non d\'eg\'en\'er\'e sur  $ \cyr{X}^{1}(K,T)$.
\end{proof}

\subsection{Un exemple.}\label{exempledetaille}

Nous pr\'esentons ici avec plus de d\'etails
un exemple donn\'e dans \cite{CTPaSu}, Prop. 3.2.
Cet exemple illustre la proposition~\ref{compatiblelementaire}
et le corollaire~\ref{elementairelaseule} ci-dessus.

\smallskip
 
Soient  $k=\C((t))$ et $K=\C((t))(x)$. On consid\`ere la $K$-vari\'et\'e  $E$
 d\'efinie par l'\'equation 

$$(X_{1}^2-xY_{1}^2)(X_{2}^2-(1+x)Y_{2}^2)(X_{3}^2-x(1+x)Y_{3}^2)=t.$$

\medskip

Assertion I.  {\it Si $P$ est un point ferm\'e de $\P^1_k$ l'un de $x, x+1, x(x+1)$ est un carr\'e dans le compl\'et\'e $K_{P}$, et donc $E(K_{P}) \neq \emptyset$ et $E_{K_{P}} $ est une vari\'et\'e $K_{P}$-rationnelle.}

  Calculons les valeurs dans $K_{P}$ pour tout point ferm\'e $P$.

Au point \`a l'infini, $x(x+1)$ est  un carr\'e. Mais ni $x$ ni $x+1$ ne le sont, car leur valuation est $-1$.

Au point $x=0$, $x+1$ est un carr\'e car en r\'eduction il vaut $1$. Ni $x$ ni $x(x+1)$ ne sont des carr\'es,
car leur valuation est $1$.

Au point $x=-1$, $x$ est un carr\'e car $-1$ est un carr\'e dans $k$. Ni $x+1$ ni $x(x+1)$ ne sont
des carr\'es, car leur valuation est $1$.

Soit $P$ un autre point ferm\'e de $\P^1_k$. Alors $x, x+1$ et $x(x+1)$ sont des unit\'es en $P$.

Soit $\xi$ la classe de $x$ dans le corps r\'esiduel $\kappa(P)$.
 On a donc $\xi\neq 0$
et $\xi \neq -1$. Soit $v$ la valuation  sur le corps r\'esiduel $\kappa(P)$, extension finie de $k=\C((t))$.

Si $v(\xi)<0$, alors $\xi(\xi+1)$ est un carr\'e dans $\kappa(P)$  et donc $x(x+1)$ est un carr\'e dans $K_{P}$.

Si $v(\xi)>0$ alors $\xi+1$ est un un carr\'e dans $\kappa(P)$ et donc $x+1$ est un carr\'e dans $K_{P}$.

Si $v(\xi)=0$ et $v(\xi+1)=0$ alors chacun de $\xi$, $\xi+1$ et $\xi(\xi+1)$ est une unit\'e  dans $\kappa(P)$.
Chacun  est  donc un carr\'e dans $\kappa(P)$. Chacun de  $x$, $x+1$,  $x(x+1)$ est donc
un carr\'e dans $K_{P}$.

Si $v(\xi)=0$ et $v(\xi+1)>0$, alors $\xi$ est 
est un carr\'e dans $\kappa(P)$
et donc   $x$ est un carr\'e dans $K_{P}$.

Ainsi pour toute valeur de $c \in K^{\times}$, en particulier pour $c=t$, l'\'equation
$$(X_{1}^2-xY_{1}^2)(X_{2}^2-(1+x)Y_{2}^2)(X_{3}^2-x(1+x)Y_{3}^2)= c$$
a des solutions dans tous les $K_{P}$ pour $P$ point ferm\'e de $\P^1_k$.

\medskip

Assertion II.  

 {\it  Soit $w$ la valuation de $K=k(\P^1)$ associ\'ee  au point g\'en\'erique de la fibre sp\'eciale de $\P^1_{\C[[t]}$.
On a $E(K_{w}) =\emptyset$, et donc $E(K)=\emptyset$.}

L'\'el\'ement $t$ est une uniformisante pour $w$, le corps r\'esiduel de $w$ est le corps $\C(x)$.  
Chacun des $x$, $x+1$, $x(x+1)$ est une $w$-unit\'e, et n'est pas un carr\'e dans le corps
r\'esiduel de $w$. On en d\'eduit que pour toute valeur $c$ non nulle de
$$(X_{1}^2-xY_{1}^2)(X_{2}^2-(1+x)Y_{2}^2)(X_{3}^2-x(1+x)Y_{3}^2)$$
avec les $X_{i}, Y_{i}$ dans $k(x)$,  on a $w(c)$ paire. Mais $w(t)=1$. Donc $t$
n'est pas repr\'esent\'e  par ce produit sur le compl\'et\'e  $K_{w}$ de $K$ en le point g\'en\'erique
de la fibre sp\'eciale, et donc pas sur $K$.

\medskip

Assertion III.  {\it Cet exemple s'explique  par la loi de r\'eciprocit\'e sur la courbe $\P^1_{k}$.}

  \medskip
  
  Notons $Z/K$   une compactification lisse de la $K$-vari\'et\'e $E$  d'\'equation
  $$(X_{1}^2-xY_{1}^2)(X_{2}^2-(1+x)Y_{2}^2)(X_{3}^2-x(1+x)Y_{3}^2)= t .$$

On sait \cite{CTBordeaux}  que le groupe de Brauer d'un mod\`ele projectif et lisse $Z$
est engendr\'e modulo $\Br(K)$ par l'\'el\'ement d'ordre 2  de $ \Br(K(Z))$ 
d\'efini par
{\small
$$A = (X_{1}^2-xY_{1}^2, x+1) =  (X_{1}^2-xY_{1}^2, x(x+1))  =$$
$$ ( X_{3}^2-x(1+x)Y_{3}^2,x+1)+ (t,x+1).$$
}

Pour $P$ parcourant les points ferm\'es de $\P^1_{k}$, on a vu  que $Z_{K_{P}}$ est
$K_{P}$-rationnel. Ceci implique $\Br Z_{K_{P}} =\Br K_{P}$.
La valeur de $A(M_{P}) \in Br(K_{P})$ ne d\'epend donc   pas du choix de $M_{P} \in Z(K_{P})$.

On a $H^1(k,\Z/2)=k^{\times}/k^{\times 2} \simeq \Z/2$. On a obstruction de r\'eciprocit\'e si l'on a :
$$\sum_{P} {\rm Cores}_{k(P)/k} (\delta_{P}(A(M_{P})) \neq 0 \in   \Z/2$$

\medskip

Au point $P$ d\'efini par $x=0$, $x+1$ est un carr\'e, donc $A_{K_{P}}=0$.

Au point  $P$ d\'efini par $x=\infty$, $x(x+1)$ est un carr\'e, donc $A_{K_{P}}=0$.

Au point d\'efini par $x=-1$, $x$ est un carr\'e, on trouve
$A_{K_{P}}=(t,x+1) \in Br(K_{P})$, et donc
$${\rm Cores}_{k(P)/k} \delta_{P}(A(M_{P})) =1 \in \Z/2.$$

Consid\'erons les   points ferm\'es $P \in \P^1_{k}$ o\`u $x$ et $x+1$ sont des unit\'es.
Soit $\xi$ la classe de $x$ dans le corps r\'esiduel $\kappa(P)$.
Soit $v$ la valuation  sur le corps r\'esiduel $\kappa(P)$, extension finie de $k=\C((t))$.

Si $v(\xi)<0$, alors $x(x+1)$ est un carr\'e dans $K_{P}$, donc $A_{K_{P}}=0$.

Si $v(\xi)>0$, alors $x+1$ est un carr\'e dans $K_{P}$, donc $A_{K_{P}}=0$.

Si $v(\xi)=0$, alors $x$ est un carr\'e dans $K_{P}$,
donc $A_{K_{P}}=(t,x+1)$, avec $t$ et $x+1$ unit\'es en $P$,
et donc $\partial_{P}(A_{K_{P}})=0$.

On obtient donc  
$$\sum_{P} {\rm Cores}_{k(P)/k} (\delta_{P}(A(M_{P})) = 1 \in \Z/2.$$
Il y a donc obstruction de r\'eciprocit\'e sur la courbe $\P^1_{k}$, et $Z(K)=\emptyset$.

\qed

Dans cet exemple, $E$ est un espace principal homog\`ene sous le $K$-tore $T$
 obtenu en rempla\c cant $t$ par $1$ dans le membre de droite de l'\'equation de $E$.
La classe de $A$ dans $\Br(Z) \subset \Br(E)$   est   constante
 en tout $v\in \Omega$. Elle d\'efinit donc un \'el\'ement de $B(E)$, dont l'image dans $\cyr{B}(E)$
 d\'efinit un \'el\'ement $\gamma
\in \cyr{X}^2(K,\widehat{T})$. La proposition \ref{compatiblelementaire} permet de  traduire
  le calcul fait ci-dessus en termes de l'accouplement
  $$ \cyr{X}^{1}(K,T) \times \cyr{X}^{2}(K,\widehat{T}) \to \Q/\Z(-1).$$

\begin{remas}

Dans l'exemple~\ref{contreexrationnel} (voir aussi  \cite[Exemple 2.6]{CTPaSu}), on a exhib\'e des
 espaces principaux homog\`enes $Y$ sous un $K$-tore $T$, avec $\prod_{v}Y(K_{v}) \neq \emptyset$
tels que
$\Br_a(Y_{c})=0$ et $Y(K)=\emptyset$.
D'apr\`es le corollaire \ref{elementairelaseule}, on doit pouvoir expliquer cet
exemple au moyen d'un \'el\'ement explicite de $B(Y)$, ou de $ \cyr{X}^{2}(K,\widehat{T})$,
mais cela ne semble pas imm\'ediat.

Les tores $T$ dans les exemples mentionn\'es sont $K$-rationnels.
Pour un tore $K$-rationnel sur un corps $K=\C((t))(C)$, c'est une cons\'equence de  th\'eor\`emes d'Harbater, Hartmann et Krashen,
et de travaux ant\'erieurs, que le principe local-global, par rapport \`a l'ensemble de toutes les valuations
discr\`etes de rang 1, vaut pour les espaces principaux homog\`enes de $K$-groupes lin\'eaires connexes
$K$-rationnels \cite[Cor. 8.11]{HHK2}. On doit donc pouvoir aussi \'etablir $Y(K)=\emptyset$
en consid\'erant un compl\'et\'e de $K$ en une place non triviale sur $\C((t))$, et c'est de fait facile
\`a \'etablir dans le cas pr\'esent.

 \end{remas}

\section{Approximation faible}

  Soient $k=\C((t))$ et $K=k(C)$ le corps des fonctions d'une courbe.
  
  Comme on a vu au d\'ebut de l'article, le corps $K$ est un corps ``de type arithm\'etique''
  au sens de \cite[\S 7]{Parimala}.  
  
  \subsection{Approximation faible pour les tores}

 Etant donn\'e $T$ un $K$-tore, on peut consid\'erer le probl\`eme
 d'approximation faible par rapport aux places $v \in C^{(1)}$.
On note $\overline{{T(K)}}$ l'adh\'erence de $T(K)$ dans le produit
direct des $T(K_v)$. De m\^eme, si $S$ est un sous-ensemble fini non vide
de $C^{(1)}$, on note $\overline{{T(K)}}_S$ l'adh\'erence de 
$T(K)$ dans $\prod_{v \in S} T(K_v)$. Le {\it d\'efaut d'approximation 
faible} pour $T$ est le conoyau de l'injection diagonale 
$\overline{{T(K)}} \to \prod_{v \in \Omega_K} T(K_v)$, ce dernier groupe 
\'etant muni de la topologie produit des topologies $v$-adiques. 

\begin{prop} \label{flasqueaf}
Soit $$1 \to Q \to R \to T \to 1$$
une r\'esolution flasque de $T$. Alors~: 

\smallskip

a) Les groupes $H^1(K_v,Q)$ sont nuls pour presque toute place $v \in C^{(1)}$, et le groupe $H^1(K,Q)$ est fini.

b) Le d\'efaut d'approximation 
faible pour $T$ est fini, et isomorphe au conoyau de la fl\`eche diagonale 
$H^1(K,Q) \to \bigoplus_{v \in \Omega_K} H^1(K_v,Q)$, qui est un homomorphisme
de groupes finis.
\end{prop}

\begin{proof}
a) Soit $v$ une place de bonne r\'eduction pour le $K$-tore flasque $Q$.
Alors $H^1(O_{v},Q)$ se surjecte sur $H^1(K_{v},Q)$ \cite[Thm. 2.2]{CTSflasque},  mais 
$H^1(O_{v},Q)=H^1(\kappa_{v},Q)=0$ car ${\rm cd}(\kappa_{v})=1$. Finalement 
$H^1(K_v,Q)$ est nul en toute place de bonne r\'eduction de $Q$. 

 Il existe un ouvert $U \subset C$ sur lequel $Q$ s'\'etend en un $U$-tore flasque, notons-le $Q$
\cite[Prop. 1.5]{CTSflasque}.

La fl\`eche $H^1(U,Q) \to H^1(K,Q)$ est surjective \cite[Thm. 2.2]{CTSflasque}. Comme $H^1(K,Q)$
est annul\'e par un entier $n>0$, la fl\`eche se factorise par $H^1(U,Q)/n$,
qui est un sous-groupe du groupe   $H^2(U,{}_{n}Q)$, qui est fini
(Lemme \ref{finitudeUfini}). Donc $H^1(K,Q)$ est fini.

\smallskip

b) Comme le tore $R$ est quasi-trivial, on a $H^1(K,R)=H^1(K_v,R)=0$ 
pour toute place $v$. On en d\'eduit, pour tout sous-ensemble fini non vide
$S$ de $C^{(1)}$, un diagramme commutatif 
\`a lignes exactes~:
$$
\begin{CD}
R(K) @>>> T(K) @>>> H^1(K,Q) @>>> 0 \cr
@VVV @VVV @VVV \cr
\prod_{v \in S} R(K_v) @>>> \prod_{v \in S} T(K_v) @>>> \prod_{v \in S} 
H^1(K_v,Q) @>>> 0
\end{CD}
$$
o\`u les groupes sur la verticale de droite sont finis.
Le tore quasi-trivial (donc $K$-rationnel) $R$ v\'erifie l'approximation 
faible. On en d\'eduit une suite exacte 
$$ 0 \to \overline{{T(K)}}_S \to \prod_{v \in S} T(K_v) \to 
\Coker [H^1(K,Q) \to \bigoplus_{v \in S} H^1(K_v,Q)].$$
Comme $H^1(K_v,Q)=0$ en dehors de l'ensemble fini $S_0$ des places 
de mauvaise r\'eduction de $Q$, on obtient le r\'esultat.
\end{proof}

\begin{rema}
{\rm Si le $K_{v}$-tore $T$ est d\'eploy\'e par une extension cyclique,
alors $H^1(K_v,Q)$=0. On peut ainsi ``diminuer'' le nombre de places de mauvaise
r\'eduction \`a consid\'erer. De fait, dans la pr\'esente 
situation, en une place de bonne r\'eduction, le tore est bien
d\'eploy\'e par une extension cyclique, car toute extension de $\C((t))$ est cyclique.}
\end{rema}

\bigskip

Soit  $F$ un module galoisien fini sur $K=\C((t))(C)$.
Il s'\'etend en faisceau \'etale $\calf$ sur un ouvert  non vide $U \subset C$.
Pour tout entier $i \geq 0$, le produit restreint  $\P^{i}(F) \subset \prod_{v \in C^{(1)}} H^{i}(K_{v},F)$
est le sous-groupe form\'e des familles $\{\xi_{v}\}$ avec $\xi_{v} \in H^{i}(O_{v},F)$ pour presque tout $v \in  U$. Pour tout module galoisien $M$, et tout ensemble fini $S$ de $C^{(1)}$,
on note $\Sha^i_S(M)$ le sous-groupe de $H^i(K,M)$ constitu\'e 
des classes nulles en dehors de $S$.

En particulier le sous-groupe 
$\Sha^i_{\omega}(M)$ des classes nulles presque partout est la limite
inductive (sur l'ensemble des $S$ finis) des $\Sha^i_S(M)$.

\begin{prop} \label{finipoit}
Soit $F$ un module galoisien fini sur $K=\C((t))(C)$, de dual 
$\widehat F$.

 Les applications de dualit\'e locale  
$$ H^1(K_{v},F) \times H^1(K_{v},\widehat F) \to \Q/\Z(-1)$$
de la proposition  \ref{dualitelocalefinie}
induisent des applications ${\bf P}^1(F) \to H^1(K,\widehat F)^D$
et $\prod_{v \in S} H^1(K_v,F) \to \Sha^1_S(\widehat F)^D$.

(a) On a une suite exacte de groupes ab\'eliens
$$ H^1(K,F) \to {\bf P}^1(F) \to H^1(K,\widehat F)^D.$$

(b) Si $S$ est un sous-ensemble fini non vide de $C^{(1)}$, on a une suite exacte 
de groupes ab\'eliens
$$H^1(K,F) \to \prod_{v \in S} H^1(K_v,F) \to \Sha^1_S(\widehat F)^D 
\to \Sha^1(\widehat F)^D \to 0. $$
\end{prop}

Le groupe $H^1(K, \widehat F)$ est en g\'en\'eral infini.
Dans la seconde suite, les groupes autres que $H^1(K,F)$ sont finis.
La finitude de $ \Sha^1_S(\widehat F)$ r\'esulte de la 
proposition \ref{extension} et du lemme \ref{finitudeUfini}.

\begin{proof} 
La d\'emonstration est enti\`erement analogue 
\`a celle de \cite[Prop. 2.1]{HaSchSz}, avec la diff\'erence qu'ici 
le corps $K$ est de dimension cohomologique 2.
 On passe par la dualit\'e 
d'Artin-Verdier (Prop. \ref{finigroup} ci-dessus) pour
les faisceaux \'etales finis sur un ouvert de la courbe $C$ pour obtenir 
le premier \'enonc\'e.
 On en d\'eduit le second en utilisant la dualit\'e
locale (Prop. \ref{dualitelocalefinie}), comme dans \cite[Lemma 3.1]{HaSchSz}.
\end{proof}

\begin{theo}
Soit $T$ un $K$-tore. 

Soit $S$ un ensemble fini non vide de places de $K$. 
Alors on a une suite exacte de groupes ab\'eliens
$$0 \to \overline{{T(K)}}_S \to \prod_{v \in S} T(K_v) \to 
\Sha^2_S(\widehat T)^D \to \Sha^2(\widehat T)^D \to 0,$$
l'application
$$\prod_{v \in S} T(K_v) \to 
\Sha^2_S(\widehat T)^D$$
 \'etant induite par les applications
de dualit\'e locale
$$T(K_{v}) \times H^2(K_{v}, \widehat T) \to \Q/\Z(-1)$$
de la proposition \ref{dualitelocaletore}. 

Dans cette suite exacte, l'image de 
$ \prod_{v \in S} T(K_v) \to 
\Sha^2_S(\widehat T)^D $ est un groupe fini.
\end{theo}

Noter que l'application $\prod_{v \in S} T(K_v) \to
\Sha^2_S(\widehat T)^D$ est continue via la remarque~\ref{topolocale}.

\begin{proof} On montre d'abord deux lemmes~:

\begin{lem} \label{lemme1}
La suite de groupes ab\'eliens
$$T(K) \to \prod_{v \in S} T(K_v) \to \Sha^2_S(\widehat T)^D$$
est un complexe. 
\end{lem}

On note d'abord qu'on a des accouplements locaux 
$$H^0(K_v^h,T) \times H^2(K_v^h,\widehat T)=H^2(K_v,\widehat T) \to 
\Br K_v^h=\Br K_v \simeq \Q/\Z(-1)$$
qui sont compatibles avec les accouplements analogues faisant 
intervenir $H^0(K_v,T)$ au lieu de $H^0(K_v^h,T)$. Il suffit donc 
de d\'emontrer que la suite 
$$T(K) \to \prod_{v \in S} T(K_v^h) \to \Sha^2_S(\widehat T)^D$$
est un complexe,
la derni\`ere fl\`eche \'etant induite par les accouplements locaux.

\smallskip

Soient $t \in T(K)$ et $\alpha \in \Sha^2_S(\widehat T)$. 
Soit $U$ un ouvert affine non vide de $C$ tel que 
$T$ s'\'etende en un $U$-tore $\calt$ avec de plus~: 
$t \in \calt(U)=H^0(U,\calt)$ et $\alpha$ se rel\`eve en un 
\'el\'ement (encore not\'e $\alpha$) de $H^2(U,\widehat \T)$.
On peut supposer (quitte \`a restreindre $U$ et \`a augmenter $S$) que 
$S=C-U$. Soient $(t_v)$ l'image de $t$ dans $\prod_{v \in S} T(K_v^h)$ et 
$\theta$ l'image de $(t_v)$ dans $\Sha^2_S(\widehat T)^D$. Alors 
$$\theta(\alpha)=\sum_{v \in S} \langle t_v , \alpha_v \rangle$$
est aussi donn\'e par l'accouplement d'Artin-Verdier 
$\langle u, \alpha \rangle$, o\`u $u$ est l'image de $(t_v)$ dans 
$H^1_c(U,\T)$ car on a compatibilit\'e entre accouplement 
d'Artin-Verdier et accouplements locaux (lemme~\ref{artlocal}).
Mais $u$ est nul via la suite de localisation  \cite[Prop. II.2.3 p. 166]{milneADT}
$$H^0(U,\T) \to \bigoplus_{v \in S} H^0(K_v ^h,T) \to H^1_c(U,\T).$$
Ceci d\'emontre le premier lemme.

\smallskip

\begin{lem} \label{lemme2}
La suite de groupes ab\'eliens
$$0 \to \overline{{T(K)}}_S \to \prod_{v \in S} T(K_v) 
\to \Sha^2_S(\widehat T)^D$$
est exacte, et l'image de $\prod_{v \in S} T(K_v) 
\to \Sha^2_S(\widehat T)^D$ est un groupe fini.
\end{lem}

L'exactitude se montre   comme dans 
le Th. 3.3. de \cite{HaSchSz}, en se ramenant (par le lemme d'Ono) 
au cas o\`u  $T$ admet une r\'esolution $$0 \to F \to R \to T \to 0$$ 
avec $F$ fini et $R$ tore quasi-trivial.
On a alors  le diagramme commutatif
\begin{equation} \label{grandiag}
\begin{CD}
R(K) @>>> T(K) @>>> H^1(K,F) @>>> 0 \cr
@VVV @VVV @VVV \cr
\prod_{v \in S} R(K_v) @>>> \prod_{v \in S} T(K_v) @>>> \prod_{v \in S} 
H^1(K_v,F) @>>> 0 \cr
@VVV @VVV @VVV \cr
 \Sha^2_S(\widehat R)^D  @>>>   \Sha^2_S(\widehat T)^D   @>>> \Sha^1_S(\widehat F)^D @>>> 0
\end{CD}
\end{equation}
o\`u les suites horizontales sont exactes, les suites 
verticales sont des complexes, et la suite verticale de droite
est exacte par la proposition~\ref{finipoit}.

 Le tore $R$ est quasi-trivial. La proposition 
 \ref{sha1sha1omega} 
montre   que l'on a  $\Sha^2_S(\widehat R)=\Sha^2(\widehat R)$.
Ceci implique que l'application
$\prod_{v \in S} R(K_v) 
\to \Sha^2_S(\widehat R)^D$
est nulle. 

Du diagramme on conclut que
  la fl\`eche
 $\prod_{v \in S} T(K_v) 
\to \Sha^2_S(\widehat T)^D$
se factorise par le groupe fini $\prod_{v\in S}H^1(K_{v},F)$.

Le tore $R$ \'etant quasitrivial satisfait l'approximation faible.
L'adh\'erence de $T(K)$ dans $ \prod_{v \in S} T(K_v)$
est le sous-groupe ouvert d'indice fini engendr\'e par $T(K)$
et l'image de $\prod_{v \in S} R(K_v) \to  \prod_{v \in S} T(K_v)$.
Notons $A_{S}(T)$ le groupe fini quotient de $ \prod_{v \in S} T(K_v)$
par l'adh\'erence de $T(K)$.

Du diagramme on d\'eduit  un homomomorphisme
$A_{S}(T) \to \Sha^2_S(\widehat T)^D$
qui compos\'e avec 
$ \Sha^2_S(\widehat T)^D \to \Sha^1_S(\widehat F)^D $
est injectif. Ainsi $A_{S}(T) \to \Sha^2_S(\widehat T)^D$
est aussi injectif.

\begin{rema}
{\rm 
Conservons les notations de la d\'emonstration. 
De la suite exacte
$$0 \to \widehat{T} \to \widehat{R} \to \widehat{F} \to 0$$
on d\'eduit une suite exacte (d\'efinissant le groupe $B$)
$$0 \to \Sha^1_{S}(\widehat{F}) \to   \Sha^2_{S}(\widehat{T})  \to   
\Sha^2_{S}(\widehat{R}) \to B \to 0,$$
et l'on v\'erifie que groupe $B$ est annul\'e par  le carr\'e de l'ordre du
groupe fini $\widehat{F}$.
Comme $\widehat{R}$ est un module de permutation, 
  la proposition~\ref{sha1sha1omega}
montre que l'on a $$  \Sha^2(\widehat{R})= \Sha^2_{S}(\widehat{R})= 
\Sha^2_{\omega}(\widehat{R})$$
et que ce groupe est divisible. Comme $B$ est annul\'e par un entier non nul, ceci implique $B=0$.
La suite  inf\'erieure dans le diagramme (\ref{grandiag}) est  donc exacte 
aussi \`a gauche.
Le groupe $ \Sha^2_{S}(\widehat{R}) ^D$ est sans torsion.

S'il en \'etait besoin, ceci permettrait de donner une 
d\'emonstration alternative 
du lemme~\ref{lemme1}, \`a savoir que la verticale m\'ediane 
dans le diagramme (\ref{grandiag})
est un complexe. Soit en effet $\xi \in T(K)$. 
Soit $\rho \in  \Sha^2_{S}(\widehat{T})^D$ 
son image par l'application compos\'ee. Comme la verticale de droite est un
 complexe,
ceci implique qu $\rho=\sigma$ avec $\sigma$ dans le sous-groupe sans torsion
$ \Sha^2_{S}(\widehat{R})^D \subset  \Sha^2_{S}(\widehat{T})^D$.
Comme $H^1(K,F)$ est annul\'e par l'ordre $n$ de $F$, on voit que
$n\xi$ est image de $\eta \in R(K)$. 
L'image de $\eta$ par l'application compos\'ee dans
 $ \Sha^2_{S}(\widehat{R})^D$ est nulle. On a donc $n\sigma=0$
 mais comme $ \Sha^2_{S}(\widehat{R})^D$ est sans torsion, $\sigma=0$
 et donc $\rho=0$.
}
\end{rema}

\begin{rema}
{\rm
On voit ainsi que l'adh\'erence du groupe $T(K)$ dans $\prod_{v\in S}T(K_v) $
a son image nulle dans $ \Sha^2_S(\widehat T)^D$.
On aurait pu \'etablir ce fait directement en rappelant  que
l'application 
 $\prod_{v\in S}T(K_v)  \to \prod_{v\in S} H^2(K_{v},\widehat T)^D$
 est continue. On obtient ainsi une d\'emonstration qui n'utilise pas
 l'\'egalit\'e $\Sha^2_S(\widehat R)=\Sha^2(\widehat R)$.
}
\end{rema}

Pour finir la preuve du th\'eor\`eme,  il faut \'etablir
l'exactitude des trois derniers termes 
$$\prod_{v \in S} T(K_v) \to
\Sha^2_S(\widehat T)^D \to \Sha^2(\widehat T)^D \to 0.$$
Celle-ci se montre comme dans le Th\'eor\`eme  3.3. de \cite{HaSchSz}.

On part de la suite exacte \'evidente de groupes discrets de torsion
$$ 0 \to \Sha^2(\widehat T) \to \Sha^2_S(\widehat T) \to \bigoplus_{v \in S} H^2(K_{v},\widehat T).$$
Par dualit\'e on obtient une suite exacte de groupes compacts
$$ \prod_{v \in S}H^2(K_{v},\widehat T)^D \to \Sha^2_S(\widehat T)^D \to  \Sha^2(\widehat T)^D \to 0,$$
o\`u les fl\`eches sont continues. 

Par dualit\'e locale, on a une application 
$$\theta : \prod_{v \in S}T(K_{v}) \to \prod_{v \in S} H^2(K_{v},\widehat T)^D$$
qui a une image dense; l'application $\theta$ est continue via la 
remarque~\ref{topolocale}.

Il suffit alors de montrer que l'image de l'application compos\'ee
$$\prod_{v \in S}T(K_{v}) \to  \prod_{v \in S}H^2(K_{v},\widehat T)^D \to \Sha^2_S(\widehat T)^D$$
co\"{\i}ncide avec celle de l'application continue
$$\prod_{v \in S}H^2(K_{v},\widehat T)^D \to \Sha^2_S(\widehat T)^D.$$
D'apr\`es le lemme \ref{lemme2},
 l'image de l'application compos\'ee ci-dessus est finie.
La densit\'e de l'image de $\theta$ assure alors
que l'image de l'application compos\'ee co\"{\i}ncide avec celle de
$ \prod_{v \in S}H^2(K_{v},\widehat T)^D \to \Sha^2_S(\widehat T)^D.$
\end{proof}

\begin{cor}\label{coropresqueV}
On a une suite exacte de groupes ab\'eliens
$$0 \to \overline{{T(K)}} \to \prod_{v \in \Omega_K} T(K_v) \to 
\Sha^2_{\omega}(\widehat T)^D \to \Sha^2(\widehat T)^D \to 0,$$
o\`u l'image de l'application $\prod_{v \in \Omega_K} T(K_v) \to 
\Sha^2_{\omega}(\widehat T)^D$ est un groupe fini. Plus pr\'ecis\'ement,
soit $S$ un ensemble fini de places contenant toutes les places
de mauvaise r\'eduction de $T$. Alors
$A_S(T)=A(T)$ et $\Sha^2_S(\widehat T)=\Sha^2_{\omega} (\widehat T)$.
\end{cor}

\begin{proof} Soit $S'$ ensemble fini de places avec $S \subset S'$.
On a alors un diagramme commutatif
$$
\begin{CD}
0 @>>> A_{S'}(T) @>>> \Sha^2_{S'}(K,\widehat T)^D @>>> 
\Sha^2(K,\widehat T)^D @>>> 0 \cr
&& @VVV @VVV @VVV \cr
0 @>>> A_{S}(T) @>>> \Sha^2_{S}(K,\widehat T)^D @>>> 
\Sha^2(K,\widehat T)^D @>>> 0
\end{CD}
$$
La fl\`eche de droite est un isomorphisme,
la fl\`eche de gauche est une surjection de groupes finis. D\`es que 
$S$ contient toutes
les places de mauvaise r\'eduction de $T$, 
la proposition \ref{flasqueaf}  montre que
la projection $A_{S'}(T) \to A_{S}(T)$ est un isomorphisme.
On conclut qu'il en est de m\^eme pour la  fl\`eche  
verticale m\'ediane, et donc
que pour $S$ contenant toutes les places de mauvaise r\'eduction pour $T$,
on a $ \Sha^2_{S}(K,\widehat T)= \Sha^2_{\omega}(K,\widehat T)$.
\end{proof}

\begin{rema} 
{\rm
Il se peut tr\`es bien que les deux groupes 
$\Sha^2_{\omega}(\widehat T)$ et $\Sha^2(\widehat T)$ soient infinis, 
mais $\Sha^2(\widehat T)$ est toujours d'indice fini dans 
$\Sha^2_{\omega}(\widehat T)$ par le corollaire~\ref{coropresqueV}. 
Les deux groupes sont \'egaux 
pour un tore quasi-trivial, et plus pr\'ecis\'ement
\'egaux si et seulement si l'approximation faible vaut pour $T$.
Noter aussi qu'on ne peut pas comme dans 
le cas o\`u $K$ est un corps de nombres identifier le dual de 
$\Sha^2(\widehat T)$ avec 
$\Sha^1(T)$ car en g\'en\'eral $\Sha^2(\widehat T)$ contient 
une partie divisible non triviale.
}
\end{rema}

\medskip

Le th\'eor\`eme suivant est le substitut, 
dans le pr\'esent contexte, de la suite exacte
de Voskresenski\v{\i} pour les tores sur un corps de nombres 
(\cite{voskbook},  Chap. 4, Section 11.6, Th. page 120, ou encore
\cite{sansuc}, Th. 9.2.ii) et \cite{CTSflasque}, Prop. 9.5).

\begin{theo} \label{vosksansuc}
Soit $K=\C((t))(C)$ comme ci-dessus, et soit $T$ un $K$-tore.
Soit $A(T)$ le groupe fini mesurant le  d\'efaut d'approximation faible et $\Sha(T)$ 
le groupe fini mesurant le d\'efaut du principe de Hasse
pour les espaces principaux homog\`enes sous $T$.
On a une suite exacte de groupes finis
$$ 0 \to A(T) \to [\Sha^2_{\omega}(K,\widehat T)/\div]^D \to \Sha^1(K,T) \to  0.$$
\end{theo}

\begin{proof}
La d\'emonstration du corollaire ci-dessus \'etablit que le groupe  $\Sha^2(K,\widehat T)$ est d'indice fini dans
$\Sha^2_{\omega}(K,\widehat T)$. Leurs sous-groupes divisibles maximaux
co\"{\i}ncident donc. Comme $A(T)$ est fini, du corollaire \ref{coropresqueV}  on 
d\'eduit une suite exacte de groupes finis
$$0 \to A(T) \to [\Sha^2_{\omega}(K,\widehat T)/\div]^D \to [\Sha^2 (K,\widehat T)/\div]^D \to 0.$$
Le th\'eor\`eme \ref{dualitesha1tore} identifie le groupe $ [\Sha^2 (K,\widehat T)/\div]^D$
avec $\Sha^1(K,T)$.
\end{proof}

\bigskip

On a une inclusion  $\Sha^2_{\cyc}(\widehat T) \subset \Sha^2_{\omega}(\widehat T)$,
donc une application 
$$\Sha^2_{\cyc}(\widehat T) \to \Sha^2_{\omega}(\widehat T)/\div$$
qui induit une application
$$[ \Sha^2_{\omega}(\widehat T)/\div]^D \to \Sha^2_{\cyc}(\widehat T)^D.$$
On a un isomorphisme entre $\Sha^2_{\cyc}(\widehat T)$ et le sous-groupe $Br_{e}T_{c}$  
du groupe de Brauer non ramifi\'e d'une compactification lisse $T_{c} $ de  $T$ form\'e des \'el\'ements nuls en l'\'el\'ement neutre.
L'application compos\'ee
$$A(T) \to \Br_{e}T_{c}^D$$
est induite par l'accouplement de r\'eciprocit\'e 
(cf. \cite{dhsza2}, preuve du Th.~6.1.).
On va voir qu'en 
g\'en\'eral cette application n'est pas injective, en construisant
maintenant un exemple qui montre que,
contrairement au cas o\`u $K$ est un corps de nombres, le d\'efaut d'approximation
faible d'un $K$-tore n'est pas toujours contr\^ol\'e par son groupe de Brauer non 
ramifi\'e (cette assertion \'etant valide une fois la compatibilit\'e ci-dessus v\'erifi\'ee).

\begin{lem}
Il existe une courbe sur $\C((t))$ de corps de fonctions $K$, telle que $K$ 
poss\`ede deux extensions galoisiennes $K_1$, $K_2$ de groupes de Galois 
respectifs $\Z/4$ et $\Z/2$, v\'erifiant~: l'extension $K_1/K$ 
est totalement d\'ecompos\'ee en tout point ferm\'e de $C$ et 
l'extension $K_2/K$ a au moins un sous-groupe de d\'ecomposition non 
cyclique.
\end{lem}

\begin{proof}
Soit $E$ la courbe elliptique sur $\C$ d'\'equation affine
$y^2=x(x-1)(x-2)$. Soit $E' \to E$ une isog\'enie de noyau $\Z/4$;
soit $X\to E$ un rev\^etement galoisien de groupe $(\Z/2)^2$, totalement 
ramifi\'e en un certain point $m \in E(\C)$, par exemple 
le rev\^{e}tement donn\'e par $\sqrt{x-x_{0}}, \sqrt{y-y_{0}}$, o\`u 
$y_{0}^2=x_{0}(x_{0}-1)(x_{0}-2)\neq 0$. Soient $K$, $K_1$, $K_2$ les
corps de fonctions respectifs de de $E \times_{\C} \C((t))$, 
$E'\times_{\C} \C((t))$, $X \times_{\C} \C((t))$. Alors $K_1/K$ est 
totalement d\'ecompos\'e en toute place de $v$ car le rev\^etement 
$E' \to E$ est \'etale au-dessus du corps alg\'ebriquement clos $\C$, 
donc totalement d\'ecompos\'e partout. De plus, en la place $v$ 
de $E \times_{\C} \C((t))$ correspondant au point ferm\'e au-dessus 
de $m$, l'extension $(K_2)_v/K_v$ est totalement ramifi\'ee, donc son 
groupe de Galois reste $(\Z/2)^2$.
\end{proof}

\begin{lem}
Il existe un module galoisien fini $M$ sur $K$ tel que le groupe 
$\Sha^1_{\omega}(M)$ contienne strictement $\Sha^1(M)+\Sha^1_{\cyc}(M)$.
\end{lem}

\begin{proof} On commence par prendre (via le lemme pr\'ec\'edent)
des extensions $K_1$, $K_2$ de $K$, de groupes
de Galois respectifs $D=\Z/4$ et $H=(\Z/2)^2$, v\'erifiant~:

a) $K_1/K$ est totalement d\'ecompos\'ee en tout point ferm\'e $v$
de $C$.

b) $K_2/K$ a au moins un sous-groupe de d\'ecomposition (disons en $v_0$)
non cyclique, i.e. \'egal \`a $H$.

\smallskip

\smallskip

Les extensions galoisiennes  $K_1$ et 
$K_2$  de $K$ sont lin\'eairement disjointes vu que $K_1$ est totalement 
d\'ecompos\'ee en $v_0$ tandis que $K_2$ est inerte.

Posons $L=K_1.K_2$, c'est une extension galoisienne de $K$ de 
groupe $G=D \times H \simeq  \Z/4 \times (\Z/2)^2$, o\`u $H=H\times 1$
est le groupe de Galois de $L$ sur $K_{1}$, et $D=1 \times D$
est le groupe de Galois de $L$ sur $K_{2}$.

Rappelons  que presque tous les groupes de d\'ecomposition  sont 
cycliques. L'hypoth\`ese a) donne que tous les groupes de d\'ecomposition 
dans $G$ sont inclus dans $H$.

Soit $M$ un $G$-module fini. Soit $\alpha \in H^1(G,M)$. Alors~:

\smallskip

i) Pour qu'on ait  $\alpha \in \Sha^1(G,M)$, il faut que la restriction 
$\alpha_H$ \`a $H$ soit nulle car le groupe de d\'ecomposition
$G_{v_0}$ est $H$. 

ii) Pour qu'on ait $\alpha \in \Sha^1_{\cyc}(G,M)$, il faut que 
$\alpha_D$ soit nulle car $D$ est un sous-groupe cyclique de $G$.

iii) Pour qu'on ait  $\alpha \in \Sha^1_{\omega}(G,M)$, il suffit  
que l'on ait  $\alpha_{D'}=0$ pour tout sous-groupe cyclique $D'$ de $H$, 
puisque pour presque toute $v$, $G_v$ est l'un de ces sous-groupes.

\smallskip

En prenant alors pour $M$ le noyau de l'augmentation $\Z/16[G] \to \Z/16$,
on obtient (cf. Sansuc \cite{sansuc}, p. 22)~: 

$$\Sha^1(G,M) \subset \Z/4 ;  \quad \Sha^1_{\cyc}(G,M) \subset \Z/4 ; \quad 
 \Sha^1_{\omega}(G,M)=\Z/8.$$
On a de fait
 $H^1(G,I_{G})=\Z/16$, le premier  et le second groupe sont
 dans $4.(\Z/16)$, et le troisi\`eme est $2.(\Z/16)$.
Ceci  montre que $\Sha^1_{\omega}(G,M)$ contient strictement le groupe
$\Sha^1_{\cyc}(G,M) + \Sha^1(G,M)$. 
On observe alors qu'on a $\Sha^1_{\cyc}(K,M) \subset \Sha^1_{\cyc}(G,M)$ car 
l'image de tout \'el\'ement de $\Sha^1_{\cyc}(K,M)$ dans 
$H^1(L,M)$ est dans $\Sha^1_{\cyc}(L,M)$, qui est nul 
(rappelons que le module galoisien $M$ est d\'eploy\'e
par $L$), donc appartient \`a $H^1(G,M) \cap \Sha^1_{\cyc}(K,M)=
\Sha^1_{\cyc}(G,M)$. Soit alors $\alpha$ un \'el\'ement de 
$\Sha^1_{\omega}(G,M)$ qui n'est pas dans $\Sha^1_{\cyc}(G,M) + \Sha^1(G,M)$,
alors $\alpha$ ne peut pas s'\'ecrire $\beta+\gamma$ avec 
$\beta \in \Sha^1_{\cyc}(K,M)$ et $\gamma \in \Sha^1(K,M)$, sinon 
$\gamma$ serait aussi dans $\Sha^1(G,M)$ d'apr\`es ce qui pr\'ec\`ede.
Finalement $\Sha^1_{\omega}(K,M)$ contient strictement 
$\Sha^1_{\cyc}(K,M) + \Sha^1(K,M)$.
\end{proof}

\begin{prop} \label{contrexwa}
Il existe un $K$-tore $T$ tel que le complexe 
$$0 \to \overline{{T(K)}} \to \prod_{v \in \Omega_K} T(K_v) \to 
\Sha^2_{\cyc}(\widehat T)^D$$ 
ne soit pas exact.
\end{prop}

\begin{proof} 
Pour tout module galoisien fini $M$, le groupe $\Sha^1_{\omega}( M)$ est fini (voir Prop. \ref{finitudeShai}).
Comme ce groupe est la r\'eunion des $\Sha^1_{S}( M)$ pour $S$ fini,
il existe un ensemble fini $S$ de places telle que  $\Sha^1_{S}(K,M)=\Sha^1_{\omega}( M)$.

Soit $M$ le module galoisien fini du lemme pr\'ec\'edent, et soit $S$ fini tel que
$\Sha^1_{S}(M)=\Sha^1_{\omega}( M)$. 
Soit $F$ le dual de $M$.  On plonge $F$ dans un tore quasi-trivial $R$ et 
on note $T:=R/F$ le tore quotient. 
On a donc une suite exacte de modules galoisiens de type fini
$$ 0 \to \widehat T \to \widehat P \to M \to 0.$$
Quitte \`a agrandir $S$,
on peut d'apr\`es ce qui a \'et\'e vu plus haut, supposer 
$\Sha^2_{S}(\widehat T)=\Sha^2_{\omega}(\widehat T)$.

On a un diagramme commutatif
$$
\begin{CD}
 \Sha^1_{\cyc}( M) @>>>  \Sha^1_{\omega}( M) \cr
 @VVV @VVV \cr
 \Sha^2_{\cyc}(\widehat T)@>>> \Sha^2_{\omega}(\widehat T)$$
\end{CD}
$$
o\`u les fl\`eches horizontales sont des inclusions.
Comme   les groupes $H^1(K,\widehat R)$, $H^1(K_v,\widehat R)$ et 
$\Sha^2_{\cyc}(\widehat R)$ sont nuls,
les fl\`eches verticales sont injectives, et $\Sha^1_{\cyc}( M)  \oi \Sha^2_{\cyc}(\widehat T) $.

Comme le groupe fini
$\Sha^1_{\omega}(M)$ contient strictement $\Sha^1(M)+\Sha^1_{\cyc}(M)$,
il existe     un \'el\'ement non 
nul $f$ de $\Sha^1_S(M)^D$ qui est non nul, mais s'annule sur 
$\Sha^1(M)+\Sha^1_{\cyc}(M)$. 

\smallskip

La suite exacte 
$$H^1(K,F) \to \prod_{v \in S} H^1(K_v,F) \to \Sha^1_S(M)^D \to 
\Sha^1(M)^D \to 0$$
(proposition~\ref{finipoit}) permet donc de relever $f$ en un 
$(f_v) \in \prod_{v \in S} H^1(K_v,F)$. Comme $H^1(K_v,R)=0$, 
ce $(f_v)$ se rel\`eve lui-m\^eme en un \'el\'ement 
$(t_v)$ de $\prod_{v \in S}
T(K_v)$. On a un diagramme commutatif  de complexes
$$
\begin{CD}
T(K) @>>> H^1(K,F) \cr
@VVV @VVV \cr
\prod_{v \in S} T(K_v) @>>> \prod_{v \in S} H^1(K_v,F) \cr 
@VVV @VVV \cr 
\Sha^2_S(\widehat T)^D @>>> \Sha^1_{S}(M)^D
\end{CD}
$$

Maintenant, on observe que~: 

a) Le compos\'e 
$$\prod_{v \in S} T(K_v) \to \Sha^2_S(\widehat T)^D = \Sha^2_{\omega}(\widehat T)^D 
\to \Sha^2_{\cyc}(\widehat T)^D$$
s'annule sur
l'\'el\'ement $(t_v)$, 
car $f$ s'annule sur 
$\Sha^1_{\cyc}(M)$, et on a aussi un isomorphisme
$\Sha^1_{\cyc}( M)  \oi \Sha^2_{\cyc}(\widehat T) $.

\smallskip

b) L'\'el\'ement $(t_v)$ n'est pas dans l'adh\'erence de $T(K)$. 
Sinon son image $(f_v) \in \prod_{v \in S} H^1(K_v,F)$ proviendrait de 
$H^1(K,F)$, et d'apr\`es la proposition~\ref{finipoit} l'image 
$f$ de  $(f_v)$ dans $\Sha^1_{\omega}(M)^D$ serait nulle, ce qui 
n'est pas le cas par construction.

\end{proof}

\begin{rema} 
{\rm A contrario, on voit facilement que si $T$ est un tore admettant 
une r\'esolution comme ci-dessus mais avec $M=\widehat F$ v\'erifiant 
$\Sha^1_{\omega}(M)=\Sha^1(M)+\Sha^1_{\cyc}(M)$, alors le complexe 
$$0 \to \overline{{T(K)}} \to \prod_{v \in \Omega_K} T(K_v) \to
\Sha^2_{\cyc}(\widehat T)^D$$
est bien exact.
}
\end{rema}

\begin{cor}
Il existe un $K$-tore flasque $Q$ tel que le complexe
$$H^1(K,Q) \to \bigoplus_{v \in \Omega} H^1(K_v,Q) \to H^1(K,\widehat Q)^D$$
ne soit pas exact.
\end{cor}

On n'a donc pas de suite exacte de Poitou-Tate ``compl\`ete'' pour un tore
dans notre cadre.

\begin{proof}
Soient $K$ et $T$ comme dans la proposition~\ref{contrexwa}. 
Consid\'erons une r\'esolution flasque 
$$1 \to Q \to R \to T \to 1 \to 0$$
de $T$. 
Notons d\'ej\`a que le tore $Q$ \'etant flasque, on a, pour $S$ 
ensemble fini de places contenant toutes les places de mauvaise 
r\'eduction 
$$H^1(K,\widehat Q)=\Sha^1_{\cyc}(\widehat Q)=\Sha^1_{\omega} (\widehat Q)=
\Sha^1_S(\widehat Q)$$
car un tore flasque qui est d\'eploy\'e par une extension cyclique est 
facteur direct d'un tore quasi-trivial. Par ailleurs on a un 
diagramme commutatif \`a lignes exactes 
$$
\begin{CD}
R(K) @>>> T(K) @>>> H^1(K,Q) @>>> 0 \cr
@VVV @VVV @VVV \cr
\prod_{v \in S} R(K_v) @>>> \prod_{v \in S} T(K_v) @>>> \prod_{v \in S}
H^1(K_v,Q) @>>> 0 \cr
&& @VVV @VVV \cr
&&   \Sha^2_{\cyc}(\widehat T)^D   @>{\simeq}>> 
\Sha^1_{\cyc}(\widehat Q)^D @>>> 0
\end{CD}
$$
et la proposition~\ref{contrexwa} nous dit qu'on peut prendre $S$ 
tel que la suite 
$$\overline{{T(K)}}_S \to \prod_{v \in S} T(K_v) \to 
\Sha^2_{\cyc}(\widehat T)^D$$
ne soit pas exacte. Comme $R(K)$ est dense dans $\prod_{v \in S} R(K_v)$,
une chasse au diagramme donne alors que la derni\`ere colonne ne 
peut pas \^etre exacte,
\end{proof}

\begin{rema}
{\rm Comme nous l'a signal\'e T. Szamuely, la th\'eorie du corps 
de classes fonctionne bien sur un corps du type $k(C)$, avec $k$ 
quasi fini, \`a condition d'avoir la condition suppl\'ementaire 
$H^1(k,J)=0$, o\`u $J$ est la jacobienne de la courbe $C$ (voir \cite{rimwhap}
et \cite{milneADT}, Chapitre I, Appendice A), 
condition qui est v\'erifi\'ee
si $k$ est fini mais pas toujours si $k=\C((t))$.
La proposition~\ref{finipoit} permet de faire le lien avec cette th\'eorie~: en 
effet cette propostion combin\'ee au th\'eor\`eme~\ref{fini} 
pour $F=\mu_n$ donne une suite exacte~:
$$K^*/K^{*^n} \to \prod^{'}_v K_v^*/K_v^{*^n} \to \Ga_K^{\rm ab}/n 
\to \Sha^2(\mu_n) \to 0$$
o\`u $\Ga_K$ est le groupe de Galois absolu de $K$ et 
$\prod^{'}$ d\'esigne le produit restreint. Soit $I_K=\prod^{'} 
K_v^*$ le groupe des id\`eles de $K$ et $C_K=I_K/K^*$ son groupe des
classes d'id\`eles. La nullit\'e 
de $H^1(k,J)$ implique celle de $\Sha^2(\mu_n)$ (propositions~\ref{cohocourbes} 
iii) et \ref{Sha2Kmun} ii)), ce qui permet sous cette hypoth\`ese d'en 
d\'eduire facilement qu'on a une application de r\'eciprocit\'e 
$C_K \to \Ga_K^{\rm ab}$ d'image dense, comme dans le cas classique 
o\`u $k$ est fini. Sans la nullit\'e de $\Sha^2(\G_m)$, on voit qu'il subsiste
une obstruction \`a la densit\'e de l'application de r\'eciprocit\'e 
(obstruction qui dispara\^{\i}t si on remplace $K$ par son 
extension ab\'elienne maximale 
totalement d\'ecompos\'ee en tous les points ferm\'es de $C$).
}
\end{rema}

\section{Groupes alg\'ebriques lin\'eaires}\label{GAL}

 Soit $K=\C((t))(C)$ le corps des fonctions d'une courbe $C$ projective, lisse,
  g\'eom\'etriquement connexe sur le corps $\C((t))$.
 On note $K_{v}$ le compl\'et\'e de $K$ en une place $v \in \Omega$, o\`u 
 $\Omega$ est l'ensemble des places associ\'ees aux points ferm\'es de $C$.

 \subsection{Rappels} \label{rapsect}
 
 Dans \cite{CTGiPa}, on a d\'ecrit de nombreuses propri\'et\'es des
 groupes alg\'ebriques lin\'eaires connexes sur divers corps de ``dimension'' 2,
 mais le cas des corps $\C((t))(C)$ n'y avait pas \'et\'e pris en consid\'eration.
 Ceci fut rectifi\'e dans l'expos\'e de Parimala \cite{Parimala}.
 Les \'enonc\'es qui suivent sont  pour l'essentiel contenus
 dans le  \S 6 et  le \S 7 de son expos\'e.
 Leur d\'emonstration utilise les propri\'et\'es
 des corps $\C((t))(C)$ rappel\'es au \S \ref{lescorps} du pr\'esent article :
 propri\'et\'e $C_{2}$, et co\"{\i}ncidence de l'exposant et de l'indice
 des alg\`ebres simples centrales sur un tel corps.

 Soit $G$ un $K$-groupe alg\'ebrique lin\'eaire connexe.

\medskip

$\bullet$  
Les quotients $G(K)/R$ et $G(K_{v})/R$ par la $R$-\'equivalence  sont finis. Si $G$ est simplement connexe,
ils sont triviaux.

$\bullet$ Pour $G$ semisimple simplement connexe sur un corps $K=\C((t))(C)$
et $S$ un ensemble fini de points ferm\'es $v$   de $C$, on a l'approximation faible :
 l'application
naturelle
$$G(K) \to \prod_{v\in S} G(K_{v})$$
a une image dense. Pour le voir, il suffit  d'adapter la d\'emonstration de \cite[Thm. 4.7]{CTGiPa}.  

$\bullet$ Si $G$ est simplement connexe, on a $H^1(K,G)=0$ et $H^1(K_{v},G)=0$.
C'est un cas particulier  de \cite[Thm. 1.2]{CTGiPa}.

$\bullet$ Un $K$-groupe  semisimple simplement connexe sans facteur de type $A_{n}$
est une $K$-vari\'et\'e $K$-rationnelle.
C'est un cas particulier de  \cite[Thm. 1.5; Thm. 4.3]{CTGiPa}.

\subsection{Obstruction au principe local-global pour les espaces principaux homog\`enes}

Le th\'eor\`eme suivant (qui s'\'etend imm\'ediatement au cas non r\'eductif
par trivialit\'e de la cohomologie galoisienne des unipotents en 
caract\'eristique z\'ero) est l'analogue du th\'eor\`eme 8.5 
de Sansuc \cite{sansuc}.

On utilise ici   les notations introduites aux  paragraphes \ref{sgpbrauer} et \ref{accouplements}.
\begin{theo}
Si $G$ est un $K$-groupe lin\'eaire connexe r\'eductif, 
l'ensemble $\cyr{X}^1(K,G)$ est fini.

On a un accouplement
$$\cyr{X}^1(K,G) \times \cyr{B}(G) \to  \Q/\Z,$$
fonctoriel et multiplicatif en $G$, 
compatible \`a toute application $\cyr{X}^1(K,G) \to \cyr{X}^2(K,\mu)$
d\'eduite d'une $K$-isog\'enie $H \to G$ de noyau $\mu$.

Cet accouplement est non d\'eg\'en\'er\'e \`a gauche. 
 Si $G$ est semisimple,
l'accouplement est non d\'eg\'en\'er\'e des deux c\^ot\'es.

Pour $E$ un espace principal homog\`ene sous $G$,
ayant des points dans tous les $K_{v}$, on a un isomorphisme naturel
$ \theta : \cyr{B}(E) \to  \cyr{B}(G) $ et  pour $\alpha \in \cyr{B}(E)$,
l'obstruction au principe de Hasse sur $E$ associ\'ee \`a $\alpha$,
soit  $\rho_{E}(\alpha) \in \Q/\Z$, 
co\"{\i}ncide avec $<E,\theta{\alpha}> \in \Q/\Z$.
\end{theo}

\begin{proof}

(esquisse) La m\'ethode \'etant analogue \`a celle 
de \cite{sansuc}, nous allons nous contenter d'en rappeler les grandes 
\'etapes. On trouvera une autre m\'ethode, bas\'ee sur l'hypercohomologie 
des complexes de tores, dans \cite{diego}; cette autre m\'ethode 
donne en outre le fait que le noyau \`a droite de l'accouplement est 
le sous-groupe divisible maximal de $\cyr{B}(G)$.

\smallskip

Une suite exacte
$$1 \to \mu \to H \times_{K} P \to G \to 1,$$
avec $H$ semisimple simplement connexe, $P$ un $K$-tore quasitrivial et
$\mu$ un $K$-groupe fini central dans $ H \times P$ d\'efinit ce que 
l'on appelle un {\it rev\^etement sp\'ecial}
d'un $K$-groupe r\'eductif $G$. 
Pour tout $K$-groupe r\'eductif $G$, il existe un
entier $n>0$ tel que $G^n$ admette un tel rev\^etement (\cite{sansuc}, 
Lemme~1.10).
 Compte tenu des annulations $H^1(K,P)=0$ ($P$ quasitrivial)
et $H^1(K,H)=0$ (vu ci-dessus),
une telle suite exacte induit une application 
$$ \cyr{X}^1(K,G) \to \cyr{X}^{2}(K,\mu)$$
dont le noyau est trivial.

D'apr\`es Sansuc 
on a un isomorphisme (\cite[Cor. 7.4]{sansuc})
$$ H^1(K,\widehat{\mu}) \oi \Ker [\Br_{a}(G) \to \Br_{a}(H\times_{K}P)]$$
et  un isomorphisme (\cite[Lemme 6.9]{sansuc})
$$   \Br_a(H\times_{K}P)]\oi H^2(K,\widehat{P}) ,$$
ceci sur tout corps $K$ et de fa\c con fonctorielle en le corps $K$.
On obtient ainsi une suite exacte
$$ 0 \to \cyr{X}^1(K,\widehat{\mu}) \to \cyr{B}(G) \to  \cyr{X}^{2}(K,\widehat{P}).$$

 Le th\'eor\`eme \ref{fini} fournit une dualit\'e de groupes finis
 $$\cyr{X}^1(K,\widehat{\mu}) \times  \cyr{X}^{2}(K,\mu) \to \Q/\Z.$$
 
 On d\'efinit alors comme dans \cite{sansuc}, Prop. 8.1)
  un accouplement
 $$ \cyr{X}^1(K,G) \times  \cyr{B}(G) \to \Q/\Z$$
 \`a la Brauer--Manin, en utilisant le fait que 
 $B(E)=B(G)$ pour $E$ un espace principal homog\`ene de $G$ 
(\cite{sansuc}, Lemme~6.8). L'accouplement ainsi d\'efini
 est compatible avec celui ci-dessus via $\cyr{X}^1(K,\widehat{\mu}) \to \cyr{B}(G)$ (preuve identique \`a celle de \cite{sansuc}, Th~8.5).

On obtient en fin de compte que l'accouplement 
$$ \cyr{X}^1(K,G) \times  \cyr{B}(G) \to \Q/\Z$$
est non d\'eg\'en\'er\'e \`a gauche. Pour $E$ un espace principal homog\`ene
sous un $K$-groupe r\'eductif, l'obstruction de Brauer-Manin associ\'ee \`a 
$\Br(E)$ (et m\^eme \`a son sous-groupe $B(E)$)
est la seule obstruction au principe de Hasse. 

\smallskip

Si enfin $G$ est suppos\'e semi-simple, alors son rev\^etement 
simplement connexe 
$$1 \to \mu \to \widetilde G \to G \to 1$$
induit, compte tenu des propri\'et\'es des corps $K$ et $K_v$ 
mentionn\'ees en \ref{rapsect}, 
une bijection $\Sha^1(K,G) \oi \Sha^2(K,\mu)$ (cas particulier de 
\cite{gonz}, Th. 5.7.i). 
D'apr\`es Sansuc (\cite{sansuc}, Th. 7.2) et compte tenu du fait 
que $\Br_a (\widetilde G)=0$ sur tout corps $K$ (car $\widetilde G$ est 
semi-simple simplement connexe), on a alors un isomorphisme 
$\Sha^1(K,\hat \mu) \oi \cyr{B}(G)$, ce qui implique que l'accouplement 
$$ \cyr{X}^1(K,G) \times \cyr{B}(G) \to \Q/\Z$$
est dans ce cas non d\'eg\'en\'er\'e des deux c\^ot\'es.

\end{proof}

\end{document}